%% file: tate-v17.tex
\documentclass{amsart}

\usepackage[mathscr]{eucal}
\usepackage{amssymb}
\usepackage[usenames,dvipsnames]{xcolor} 
\usepackage[normalem]{ulem}
\usepackage{amsthm}
\usepackage{mathtools}
\usepackage{bbold}
\usepackage{mathdots}
\usepackage{float}
\usepackage{enumerate}
\usepackage{amsmath}
\usepackage{yhmath}
\usepackage{wasysym}
\usepackage[all]{xy}
\SelectTips{cm}{}
\newdir{ >}{{}*!/-10pt/\dir{>}}
\usepackage{tikz}
\usetikzlibrary {decorations.pathmorphing}
\usetikzlibrary{arrows}
\usetikzlibrary{calc}
\usepackage{tikz-cd}

\usepackage[unicode]{hyperref} 

\hypersetup{colorlinks=true,citecolor=Brown,urlcolor=Brown,linkcolor=Brown}

\usepackage{microtype}
\usepackage[nameinlink,capitalise,noabbrev]{cleveref}

\numberwithin{equation}{section}
\swapnumbers 

\newtheorem*{Thm*}{Theorem}
\newtheorem{Claim}[equation]{Claim}
\newtheorem{Cor}[equation]{Corollary}
\newtheorem{Lem}[equation]{Lemma}
\newtheorem{Prop}[equation]{Proposition}
\newtheorem{Thm}[equation]{Theorem}

\theoremstyle{remark}
\newtheorem{Def}[equation]{Definition}
\newtheorem{Exa}[equation]{Example}
\newtheorem{Hyp}[equation]{Hypothesis}
\newtheorem{Not}[equation]{Notation}
\newtheorem{Que}[equation]{Question}
\newtheorem{Rec}[equation]{Recollection}
\newtheorem{Rem}[equation]{Remark}

\Crefname{Claim}{Claim}{Claims}
\Crefname{Cor}{Corollary}{Corollaries}
\Crefname{Lem}{Lemma}{Lemmas}
\Crefname{Prop}{Proposition}{Propositions}
\Crefname{Thm}{Theorem}{Theorems}

\Crefname{Def}{Definition}{Definitions}
\Crefname{Exa}{Example}{Examples}
\Crefname{Hyp}{Hypothesis}{Hypotheses}
\Crefname{Not}{Notation}{Notations}
\Crefname{Que}{Question}{Questions}
\Crefname{Rec}{Recollection}{Recollections}
\Crefname{Rem}{Remark}{Remarks}

\AddToHook{env/Claim/begin}{\crefalias{equation}{Claim}}
\AddToHook{env/Cor/begin}{\crefalias{equation}{Cor}}
\AddToHook{env/Lem/begin}{\crefalias{equation}{Lem}}
\AddToHook{env/Prop/begin}{\crefalias{equation}{Prop}}
\AddToHook{env/Thm/begin}{\crefalias{equation}{Thm}}

\AddToHook{env/Def/begin}{\crefalias{equation}{Def}}
\AddToHook{env/Exa/begin}{\crefalias{equation}{Exa}}
\AddToHook{env/Hyp/begin}{\crefalias{equation}{Hyp}}
\AddToHook{env/Not/begin}{\crefalias{equation}{Not}}
\AddToHook{env/Que/begin}{\crefalias{equation}{Que}}
\AddToHook{env/Rec/begin}{\crefalias{equation}{Rec}}
\AddToHook{env/Rem/begin}{\crefalias{equation}{Rem}}

\newcommand{\nc}{\newcommand}
\nc{\dmo}{\DeclareMathOperator}

\nc{\beren}[1]{{\color{MidnightBlue}#1}}
\nc{\paul}[1]{{\color{ForestGreen}#1}}

\dmo{\Ab}{Ab}
\dmo{\CAlg}{CAlg}
\dmo{\cname}{c}
\dmo{\colim}{colim}
\dmo{\cone}{cone}
\dmo{\Der}{D}
\dmo{\DMack}{DMack}
\dmo{\dname}{d}
\dmo{\DRep}{DRep}
\dmo{\Excs}{Exc}
\dmo{\Fun}{Fun}
\dmo{\gen}{gen}
\dmo{\height}{height}
\dmo{\Ho}{Ho}
\dmo{\hocolim}{hocolim}
\dmo{\Hom}{Hom}
\dmo{\homol}{h}
\dmo{\id}{id}
\dmo{\Id}{Id}
\dmo{\Img}{Im}
\dmo{\incl}{incl}
\dmo{\Ind}{Ind}
\dmo{\Inj}{Inj}
\dmo{\Ker}{Ker}
\dmo{\Komp}{K}
\dmo{\Kos}{Kos}
\dmo{\Locname}{Loc}
\dmo{\Nil}{Nil}
\dmo{\opname}{op}
\dmo{\perf}{perf}
\dmo{\rank}{rank}
\dmo{\rmH}{H}
\dmo{\SH}{SH}
\dmo{\Sp}{Sp}
\dmo{\Spc}{Spc}
\dmo{\Spec}{Spec}
\dmo{\Sphere}{\mathbb{S}}
\dmo{\stab}{stab}
\dmo{\Stab}{Stab}
\dmo{\supp}{supp}
\dmo{\Supp}{Supp}
\dmo{\Thick}{Thick}
\dmo{\triv}{triv}

\nc{\adh}[1]{\mkern 1.5mu\overline{\mkern-1.5mu#1\mkern-1.5mu}\mkern 1.5mu}
\nc{\adhpt}[1]{\adh{\{#1\}}}
\nc{\adj}{\dashv}
\nc{\aka}{{a.\,k.\,a.}\ }
\nc{\bbF}{\mathbb{F}}
\nc{\bbN}{\mathbb{N}}
\nc{\bbP}{\mathbb{P}}
\nc{\bbZ}{\mathbb{Z}}
\nc{\bG}{b_G}
\nc{\cA}{\cat{A}}
\nc{\cat}[1]{\mathscr{#1}}
\nc{\cB}{\cat{B}}
\nc{\cc}{^{\cname}}
\nc{\cJ}{\cat{J}}
\nc{\cK}{\cat{K}}
\nc{\cL}{\cat{L}}
\nc{\cM}{\cat{M}}
\nc{\compunit}{\widehat{\unit}}
\nc{\cP}{\cat{P}}
\nc{\cQ}{\cat{Q}}
\nc{\cR}{\cat{R}}
\nc{\cS}{\cat{S}}
\nc{\cT}{\cat{T}}
\nc{\Db}{\Der^{\mathrm{b}}}
\nc{\dd}{^{\dname}}
\nc{\Dperf}{\Der^{\perf}}
\nc{\ee}{\mathbb{e}}
\nc{\eg}{{\sl e.\,g.}}
\nc{\Exc}[1]{\Excs_{#1}}
\nc{\eY}{\ee_Y}
\nc{\eYfinv}{\ee_{f^{-1}(Y)}}
\nc{\eZ}{\ee_Z}
\nc{\ff}{\mathbb{f}}
\nc{\fhat}[1]{\wideparen{#1}}
\nc{\fhomol}{f^{\hspace{0.5pt}\homol}}
\nc{\Fp}{{\bbF_{\hspace{-0.1em}p}}}
\nc{\fY}{\ff_Y}
\nc{\fYfinv}{\ff_{f^{-1}(Y)}}
\nc{\fZ}{\ff_Z}
\nc{\gm}{\mathfrak{m}}
\nc{\hcomp}{\pi}
\nc{\HFp}{{\rmH \hspace{-0.15em}\bbF_{\hspace{-0.1em}p}}}
\nc{\Hk}{{\rmH \hspace{-0.1em}k}}
\nc{\Homcat}[1]{\Hom_{\cat #1}}
\nc{\hook}{\hookrightarrow}
\nc{\hSd}{\hat{\cS}\dd}
\nc{\hTd}{\hat{\cT}^{\,\dname}}
\nc{\hTuYd}{(\hat{\cT}^{\hspace{0.1ex}Y})\dd}
\nc{\ideal}[1]{\langle #1\rangle}
\nc{\ie}{{\sl i.e.}, }
\nc{\into}{\mathop{\rightarrowtail}}
\nc{\inv}{^{-1}}
\nc{\isoto}{\mathop{\overset{\sim}\to}}
\nc{\Kac}{\Komp_{\textrm{ac}}}
\nc{\KInj}{\Komp\Inj}
\nc{\Loc}[1]{\Locname\langle#1\rangle}%
\nc{\Loco}[1]{\Locname_{\otimes}\hspace{-0.3ex}\langle#1\rangle}%
\nc{\Mid}{\,\big|\,}
\nc{\Mod}{\mathrm{Mod}\textrm{\rm-}}%
\nc{\Modd}{\mathrm{Mod}}
\nc{\mmod}{\textrm{\rm-}\mathrm{mod}}%
\nc{\num}[1]{[#1]}
\nc{\onto}{\mathop{\twoheadrightarrow}}
\nc{\op}{^{\opname}}
\nc{\overbar}[1]{\mkern 1.5mu\overline{\mkern-1.5mu#1\mkern-1.5mu}\mkern 1.5mu}
\nc{\potimes}[1]{^{\otimes #1}}
\nc{\qquadtext}[1]{\qquad\textrm{#1}\qquad}
\nc{\quadtext}[1]{\quad\textrm{#1}\quad}
\nc{\restr}[1]{|_{#1}}
\nc{\sbull}{{\scriptscriptstyle\bullet}}
\nc{\Sc}{\cS\cc}
\nc{\Sd}{\cS\dd}
\nc{\SET}[2]{\big\{\,#1\Mid#2\,\big\}}
\nc{\SHd}{\SH^{\kern.05em\dname}}
\nc{\SHdp}{{\SHd_{(p)}}}
\nc{\SHEnd}{\SH_{E(n)}^{\kern.05em\dname}}
\nc{\SHGd}{\SH(G)^{\kern.05em\dname}}
\nc{\SHKnd}{\SH_{K(n)}^{\kern.05em\dname}}
\nc{\sminus}{\smallsetminus}
\nc{\Spch}{\Spc^{\homol}}
\nc{\SpchT}{\Spch(\Td)}
\nc{\SpcS}{\Spc(\cat S\dd)}
\nc{\SpcT}{\Spc(\Td)}
\nc{\Spech}{\Spec^{\homol}}
\nc{\Supph}{\Supp^{\homol}}
\nc{\SZ}{\cS_{Z}}
\nc{\SZp}{\SZ^{\perp}}
\nc{\SZpp}{\SZ^{\perp\perp}}
\nc{\Tc}{\cT^{\kern.1em\cname}}
\nc{\Td}{\cT^{\kern.1em\dname}}
\nc{\td}{\ttt_d}
\nc{\tG}{\ttt_G}
\nc{\tH}{\ttt_H}
\nc{\then}{\Rightarrow}
\nc{\ttt}{\mathbb{t}}
\nc{\TU}{\cT\restr{U}}
\nc{\TUd}{\TU^{\hspace{0.09em}\dname}}
\nc{\TY}{\cT_Y}%
\nc{\tY}{\ttt_Y}
\nc{\TYc}{\TY^{\kern0.1em\cname}}
\nc{\TYp}{\TY^{\perp}}
\nc{\TYpp}{\TY^{\perp\perp}}
\nc{\unit}{\mathbb{1}}
\nc{\unitS}{\unit_{\cat S}}
\nc{\unitT}{\unit_{\cat T}}
\nc{\vbuff}{\vphantom{I^{I}_J}}
\nc{\xto}[1]{\xrightarrow{#1}}
\nc{\Ycons}{{\overbar{Y}}^{\mathrm{cons}}}
\nc{\Yinv}{{\overbar{Y}}^{\mathrm{inv}}}

\renewcommand{\complement}{\mathsf{C}}
\renewcommand{\emptyset}{\varnothing}

\renewcommand{\mod}{\mathrm{mod}\textrm{\rm-}}%

\newcounter{enum-resume-hack}

\title{The Tate Intermediate Value Theorem}

\author[P.\ Balmer]{Paul Balmer}
\address{UCLA Mathematics Department, Los Angeles, CA 90095-1555, USA}
\email{balmer@math.ucla.edu}
\urladdr{https://math.ucla.edu/$\sim$balmer/}

\author[B.\ Sanders]{Beren Sanders}
\address{Mathematics Department, UC Santa Cruz, 95064 CA, USA}
\email{beren@ucsc.edu}
\urladdr{https://people.ucsc.edu/$\sim$beren/}

\date{2026 November 30}



\thanks{The first author was supported by NSF grant DMS-215375.}

\begin{document}


\maketitle

\begin{abstract}
	\vspace*{-3em}	We explain how the gluing of a closed piece of the tensor-triangular spectrum with its open complement hinges on the support of the Tate ring.
\end{abstract}

{
\hypersetup{linkcolor=black}
\tableofcontents
}

\section{Introduction}

Let $\cat T$ be a rigidly-compactly generated tensor-triangulated category. Write~$\Td$ for the subcategory of `small' objects in~$\cT$, namely the dualizable objects or, equivalently, the compact ones. The premier invariant of tensor triangular geometry is the spectrum~$\SpcT$. This topological space yields the classification of all tt-ideals of~$\Td$. It also provides information about the support of `big' objects of~$\cT$. This explains why the space~$\SpcT$ has been computed, is being computed, and will be computed in a wide range of examples.

When attempting such computations, it is standard to decompose~$\SpcT$ into subspaces, like the images of maps~$\SpcS\to\SpcT$ induced by auxiliary \mbox{tt-functors} $\cT\to\cS$. In doing so, we often encounter a partition of the spectrum
	\begin{equation}
	\label{eq:UV}%
		\SpcT=U\sqcup Y
	\end{equation}
into disjoint subsets~$U$ and~$Y$ whose topology, as subspaces of~$\SpcT$, we understand independently. To fix the ideas, assume that $Y$ is closed and that $U$ is quasi-compact open. (For specialists: it suffices for $Y$ to be a Thomason subset.)

In such a situation, the \emph{set}~$\SpcT$ is completely determined by~\eqref{eq:UV} and a great deal of its topology reduces to that of the subspaces~$U$ and~$Y$. Indeed, the topology is often characterized by \emph{specializations}~$x\leadsto y$. This notation, borrowed from algebraic geometry, indicates that $y\in\adhpt{x}$. We say that $y$ is a specialization of~$x$, or that $x$ a generalization of~$y$. If both $x$ and~$y$ belong to~$U$, or if they both belong to~$Y$, the information about $x\leadsto y$ comes from our knowledge of the spaces~$U$ and~$Y$. Since $U$ is generalization-closed and $Y$ is specialization-closed, the only mystery is:
\begin{Que}
\label{Que:specialization}%
	\emph{For $x$ in~$U$ and~$y$ in~$Y$, when do we have $x\leadsto y$?}
\end{Que}
To appreciate the problem, the reader should imagine the case where both~$U$ and~$Y$ are at the same time closed and open, so that $\SpcT$ is a topological coproduct of~$U$ and~$Y$; in this situation the answer to \Cref{Que:specialization} is `Never'. This happens exactly when $\cT=\cT_1\times \cT_2$ is a product of two tt-categories, such that $U=\Spc(\cT_1\dd)$ and~$Y=\Spc(\cT_2\dd)$. See \Cref{Rem:split}. The mere decomposition~\eqref{eq:UV} cannot differentiate this split case from a non-split case. The spaces~$U$ and~$Y$, in isolation, cannot remember the full topology of~$\SpcT$. We need some information that bridges the gap between~$U$ and~$Y$, in order to answer~\Cref{Que:specialization}.

To this end, we invoke the \emph{Tate ring}~$\tY$ associated to~$Y$. This construction is due to Greenlees~\cite{Greenlees01} and is recalled in \cref{Def:Tate} below. The ring object
	\[
		\tY = \compunit\restr{U}
	\]
is obtained in~$\cT$ from the unit~$\unit$ of the tensor product by completing~$\smash{\widehat{(-)}}$ along the closed subset~$Y$ and then localizing~$(-)\restr{U}$ onto the open complement~$U$. The support of~$\tY$ (as a `big' object, see~\Cref{Def:Supp}) is entirely contained in~$U$. We can now state our \emph{Tate Intermediate Value Theorem}:
\begin{Thm}[\cref{Cor:TIV}]
\label{Thm:TIV-intro}%
	Let $Y\subset\SpcT$ be a Thomason subset (\eg\ a closed subset with quasi-compact complement) and let $x\in\SpcT\sminus Y$ and $y\in Y$ be two points, one outside of~$Y$ and one inside~$Y$. We have a specialization~$x\leadsto y$ in~$\SpcT$ if and only if there exists an intermediate point~$z\in\Supp(\tY)$ in the support of the Tate ring such that $x\leadsto z\leadsto y$.
\end{Thm}
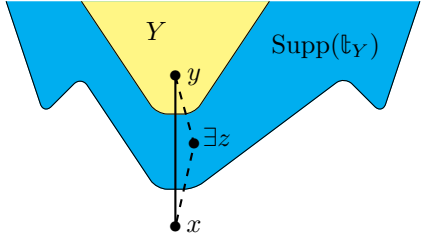
	\[
	  \begin{tikzpicture}
		\filldraw[color=cyan,rounded corners] (-1,0) -- (-.5,-1.5) -- (0,-1.0)--(1.0,-2.5) --(1.5,-2.5) -- (3.5,-1.0)-- (4,-1.5)-- (4.5,0);
		\draw[color=black,rounded corners] (-1,0) -- (-.5,-1.5) -- (0,-1.0)--(1.0,-2.5) --(1.5,-2.5) -- (3.5,-1.0)-- (4,-1.5)-- (4.5,0);
		\filldraw[color=yellow!60,rounded corners] (0,0) -- (1,-1.5) --(1.5,-1.5) -- (2.5,0);
		\draw[color=black,rounded corners] (0,0) -- (1,-1.5) --(1.5,-1.5) -- (2.5,0);
		\node at (1,-0.4) {$Y$};
		\draw (1.25,-1) node {$\bullet$};
		\draw[thick] (1.25,-1) -- (1.25,-3);
		\node at (1.5,-1) {$y$};
		\draw (1.5,-1.9) node {$\bullet$}; \node at (1.8,-1.8) {$\exists z$}; \draw[thick,dashed] (1.25,-1) -- (1.5,-1.9) -- (1.25,-3);
		\node at (3.25,-.6) {$\Supp(\tY)$};
		\draw (1.25,-3) node {$\bullet$}; \node at (1.5,-3) {$x$};
	  \end{tikzpicture}
	\]
This result explains how $Y$ is attached to its complement in $\Spc(\Td)$. For instance, the split case mentioned above (that is, $Y$ open and closed) is exactly the situation where~$\tY=0$, in which case $\Supp(\tY)=\varnothing$ and no point~$z$ can exist.

Note that the complement $U=\SpcT\sminus Y$ may be identified with the spectrum of the localization $\cat T \to \TU:=\cT/\TY$ away from~$Y$. On the other hand, we will see that $Y \cup \Supp(\tY)$ is precisely the image of the map on spectra induced by the completion $\Td \to \hTd$ along~$Y$. The overlap between those two pieces is precisely~$\Supp(\tY)$. The `bottom' specialization $x \leadsto z$ in \Cref{Thm:TIV-intro} holds in the spectrum of the localization~$\TU$. On the other hand, it follows from a strong form of the Tate Intermediate Value Theorem (\cref{Thm:TIV}) that one can find $z$ so that the `top' specialization $z \leadsto y$ is also the image of such a specialization in the spectrum of the completion~$\hTd$.

We can consider the commutative square of tt-categories
	\begin{equation}
	\label{eq:pb-square}%
	\vcenter{\xymatrix{
	\Td \ar[r]^-{\widehat{(-)}} \ar@{->>}[d]_-{(-)\restr{U}}
	& \hTd \ar@{->>}[d]^-{(-)\restr{V}}
	\\
	\Td\restr{U} \ar[r]
	& \hTd\restr{V}
	}}
	\end{equation}
comparing completion along~$Y$, in the top row, with the respective localizations on~$U$ and on $V=\varphi\inv(U)$, where~$\varphi\colon\Spc(\hTd)\to \SpcT$ is the map induced by completion. We prove in \cref{Cor:pushout} that the induced square on tt-spectra is a pushout. This explains how $\SpcT$ is recovered from~$U=\Spc(\Td\restr{U})$ and $\Spc(\hTd)$ glued along $V$. This pushout property nicely matches a recent result of Naumann--Pol--Ramzi~\cite[Theorem~5.11]{NaumannPolRamzi24} stating that the `fracture' square of $\infty$-categories underlying \eqref{eq:pb-square} is cartesian in a suitable $\infty$-category of $\infty$-categories.

Our set-up is extremely general: Tensor-triangulated categories are abundant throughout mathematics and their spectra offer an unlimited supply of Thomason subsets~$Y$ to play with. It is therefore unsurprising that we can give a long list of examples that illustrate our results. The simpler of these examples are sprinkled over the first few sections. The less elementary ones are concentrated in the last section of the paper, which is entirely dedicated to examples and applications.


\subsection*{Outline of the paper:}
We briefly recall some standard definitions and terminology in \cref{sec:basics} including notions of support for big objects in tt-geometry. We use this support to give a description of the image of a map on spectra which may be of independent interest (\cref{Thm:Img=Supp(A)}). In \cref{sec:recollement} we recall the recollements of $Y$-torsion, $U$-local and $Y$-complete subcategories and we clarify the meaning of completion along~$Y$. We discuss the completion map~$\varphi$ on the spectrum in~\Cref{sec:Spc-completion}.
Among other things, we prove that $\varphi$ is a bijection above~$Y$ (\Cref{Thm:Spc(compl)-on-Y}).
Then in \cref{sec:functorial} we make preparations for the proof of the main theorem by analyzing the interaction between completion and geometric tt-functors, in particular localizations.
We also prove in \Cref{Thm:Tate-excision} that the Tate construction satisfies excision around~$Y$.
In \cref{sec:TIV} we state and prove the main theorem and we give the description of~$\SpcT$ as a pushout in \cref{sec:topology}. We close the paper in \cref{sec:examples} with a survey of examples arising in chromatic homotopy theory (\cref{Thm:p-completion} and \cref{Prop:Kn}), equivariant stable homotopy theory (\cref{Thm:equiv}), modular representation theory (\cref{Prop:DMack}) and Goodwillie calculus (\cref{Thm:excisive}).

\section{Support of big objects}
\label{sec:basics}%

A standard reference for triangulated categories is Neeman's textbook~\cite{Neeman01}. The original reference for tt-geometry is~\cite{Balmer05a}.

\begin{Rec}
\label{Rec:big}%
	A \emph{rigidly-compactly generated tensor-triangulated category} $\cat T$, or \emph{a `big' tt-category} for short, is a triangulated category with a compatible closed symmetric monoidal structure in the sense of~\cite[Appendix~A]{HoveyPalmieriStrickland97}, which admits arbitrary coproducts, which is compactly generated as a triangulated category, and which has the property that its compact objects coincide with its dualizable objects. The latter are the `small' objects of~$\cT$ and they form an essentially small tt-subcategory~$\Td \subset \cat T$. By construction, the tt-category $\Td$ is \emph{rigid}, meaning that every object is dualizable.
\end{Rec}

\begin{Not}
	We write $\otimes$ for the tensor, $[-,-]$ for the internal hom and $\unit$ for the $\otimes$-unit. For $c \in \Td$ dualizable we write $c^{\vee} \coloneqq [c,\unit]$ for its dual.
\end{Not}

\begin{Rec}
\label{Rec:geometric}%
	A \emph{geometric functor}~$f^*\colon \cT \to \cS$ between `big' tt-categories is a tensor-triangulated functor that commutes with arbitrary coproducts. Such a functor admits a right adjoint~$f_*\colon\cS\to \cT$ which itself admits a further right adjoint $f^!\colon \cT \to \cS$; see~\cite{Neeman96} and~\cite[Section 2]{BalmerDellAmbrogioSanders16}. We also have $f^*(\Td)\subseteq \cat \Sd$ since tensor functors preserve dualizable objects. Thus we have the restricted tt-functor $f^*\colon \Td \to \Sd$ and a continuous map $f\coloneqq\Spc(f^*)\colon \Spc(\Sd)\to\Spc(\Td)$.
\end{Rec}

\begin{Rec}
\label{Rec:mod}%
	For an essentially small additive category~$\cK$, like~$\cK=\Td$, a (right) \emph{$\cK$-module} is an additive functor $\cK\op\to \Ab$ to abelian groups. We denote by~$\Mod\cK$ the abelian category of all such modules and by $\mod\cK$ the subcategory of \emph{finitely presented} ones, \ie the cokernels of maps of~$\cK$ realized in~$\Mod\cK$ via Yoneda. See more in~\cite[Sections 1--2]{Krause00} or~\cite[Chapter~5]{Neeman01}.

	If $\cK$ is triangulated then the Yoneda embedding~$\homol\colon\cK\to \mod\cK$ is the universal homological functor to an abelian category~\cite[Lemma~2.1]{Krause00} and~$\homol\colon\cK\to \Mod\cK$ is the universal homological functor to an abelian category satisfying~(AB\,5)~\cite[Lemma~2.2]{Krause00}. Hence any triangulated functor ${f^*\colon \cK\to \cL}$ to another essentially small triangulated category induces a unique \emph{exact} coproduct-preserving functor $\tilde{f}^*\colon \Mod\cK\to \Mod\cL$ that also preserves finite presentation~$\tilde f^*(\mod\cK)\subseteq\mod\cL$ and agrees with~$f^*$ under Yoneda: $\homol\circ f^*\cong \tilde{f}^*\circ\homol$.

	If $\cK$ is \emph{tensor-}triangulated then $\Mod\cK$ inherits a unique tensor product (Day convolution) that is right-exact and commutes with coproducts in both variables, preserves $\mod\cK\times\mod\cK\to \mod\cK$ and makes the Yoneda embedding into a tensor functor. Moreover, this tensor makes the induced functor~$\tilde{f}^*\colon \Mod\cK\to \Mod\cL$ into a tensor functor whenever $f^*\colon \cK\to \cL$ is a tt-functor.
\end{Rec}

\begin{Rec}
	A \emph{weak ring} in~$\cT$ is an object~$A$ with a morphism~$\eta\colon \unit\to A$ such that $A\otimes \eta\colon A\to A\otimes A$ is a split monomorphism. Of course, ring objects and their unit map~$\eta$ define weak rings, since multiplication is a retraction of~$A\otimes\eta$.
\end{Rec}

\begin{Rec}
\label{Rec:supp}%
	The support $\supp(c)\subseteq\SpcT$ of a small object~$c\in\Td$ is built into the definition of the spectrum. For big objects~$t\in\cT$ several notions of support are conceivable. The first general one was proposed in~\cite{BalmerFavi11} by means of idempotents but the construction needs the space~$\Spc(\Td)$ to be weakly noetherian; see also~\cite{BarthelHeardSanders23b}. A more general notion of support for big objects is the \emph{homological support}~$\Supph(t)$ of~\cite{Balmer20_bigsupport}. It is a subset of the \emph{homological spectrum}
	\begin{equation}
	\label{eq:Spch}%
		\Spch(\Td)\coloneqq\{\cB\subsetneq \mod\Td \textrm{ maximal Serre $\otimes$-ideal}\}.
	\end{equation}
	For every `homological prime'~$\cB$ in~$\Spch(\Td)$, we have a homological $\otimes$-functor
	\[
		\xymatrix{\homol_{\cB}\colon \cT \ar[r]^-{\homol} & \Mod\Td  \ar[r]^-{Q_{\cB}} & (\Mod\Td )/\Loc{\cB}=:\bar{\cA}_{\cB}}
	\]
	called the `homological residue field' functor at~$\cB$; it is composed of the restricted Yoneda functor~$\homol$ followed by the Gabriel quotient~$Q_{\cB}$ of the big module category~$\Mod\Td$ by the localizing ideal $\Loc{\cB}=\Loco{\cB}$ generated by~$\cB$.
\end{Rec}

\begin{Def}\label{Def:hcomp}%
	The homological spectrum comes with a surjective comparison map
	\[
		\hcomp\colon \Spch(\Td)\onto \SpcT
	\]
	defined by~$\hcomp(\cB)\coloneqq\homol\inv(\cB)=\Ker\big((\homol_{\cB})|_{\Td}\big)$. This map~$\hcomp$ is a bijection in all known examples. When $\pi$ is bijective, we say that $\cat T$ satisfies the \emph{steel condition}.
\end{Def}

\begin{Def}
\label{Def:Supp}%
	For every `big' object~$t\in\cT$, one defines $\Supph(t)$ in~$\SpchT$ as $\SET{\cB\in\SpchT}{[t,E_{\cB}]\neq0}$ where~$E_\cB$ is the pure-injective object in $\cat T$ corresponding to the injective hull of~$\unit$ in the residue category~$\bar{\cA}_{\cB}$. Conveniently, for a weak ring $A$, this definition simplifies into the following conceptually clearer formula:
	\begin{equation}\label{eq:supph-wring-2}
	\begin{array}{rl}
		\Supph(A)\kern-.7em&=\SET{\cB\in\Spch(\Td)}{\homol_{\cB}(A)\neq 0}.
	\end{array}
	\end{equation}
	See~\cite[Theorem~4.7]{Balmer20_bigsupport}. One can then define the `big' support of an object $t\in\cT$ in the ordinary spectrum~$\SpcT$ by projecting $\Supph(t)$ along~$\hcomp$:
	\begin{equation}\label{eq:Supp(t)}
		\Supp(t) \coloneqq \pi(\Supph(t)) \subseteq \SpcT.
	\end{equation}
	For a weak ring~$A$ this reads
	\begin{equation*}
		\Supp(A) =\SET{\cP\in\SpcT}{\exists\,\cB\in \hcomp\inv(\{\cP\})\textrm{ such that }\homol_{\cB}(A)\neq 0}.
	\end{equation*}
	If~$\cT$ satisfies the steel condition,~$\hcomp\colon\SpchT\isoto\SpcT$, then the above quantifiers simplify: The weak ring $A$ is supported at~$\cP$ if and only if~$\homol_{\cP}(A)$ is non-zero, where $\homol_{\cP}=\homol_{\pi\inv(\cP)}$ is the \emph{unique} homological residue field functor at the point~$\cP$.
\end{Def}

\begin{Rem}
	When $\Spc(\Td)$ is weakly noetherian, the above support $\Supp(A)$ agrees with the Balmer--Favi support if $A$ is a weak ring; see~\cite[Proposition~5.11]{BarthelHeardSandersZou26}. This agreement can fail for general objects; see~\cite[Example~5.5]{BarthelHeardSanders23b}.
\end{Rem}

\begin{Rem}\label{rem:detection-for-wring}
	If $A$ is a weak ring then $\Supp(A) = \emptyset$ implies $A=0$ by~\cite[Theorem 1.8]{Balmer20_bigsupport}. Again, this `detection property' need not hold for general objects.
\end{Rem}

\begin{Rem}\label{rem:unital-map}
	If $A \to A'$ is a unital map of weak rings then $\homol_{\cB}(A)\to \homol_{\cB}(A')$ remains a unital map in~$\bar{\cA}_{\cB}$, for every homological prime~$\cB$. We deduce from~\eqref{eq:supph-wring-2} that $\Supph(A) \supseteq \Supph(A')$ and hence $\Supp(A) \supseteq \Supp(A')$, too.
\end{Rem}

\begin{Rem}
	Let $f^*:\cat T \to \cat S$ be a geometric functor and let
	$f \coloneqq \Spc(f^*):\Spc(\Sd)\to\Spc(\Td)$
	and
	$\fhomol \coloneqq \Spch(f^*):\Spch(\Sd)\to\Spch(\Td)$
	be the induced maps. These fit into a commutative diagram
	\begin{equation}\label{eq:pi-natural}
		\begin{tikzcd}
			\Spch(\Sd) \ar[d,two heads, "\hcomp"']\ar[r,"\fhomol"] & \Spch(\Td)\ar[d,two heads,"\hcomp"] \\
			\Spc(\Sd) \ar[r,"f"] & \Spc(\Td).
	\end{tikzcd}
	\end{equation}
	The following base-change formulas for support are established in~\cite{BarthelHeardSandersZou26}.
\end{Rem}

\begin{Prop}\label{Prop:Supp-base-change}%
	With the above notation, the following statements hold:
	\begin{enumerate}[\rm(a)]
		\item\label{it:Supp-a} For any weak ring $A$ in $\cat T$, we have $\Supph(f^*(A))=(\fhomol)^{-1}(\Supph(A))$.
			If~$\cat T$ satisfies the steel condition (\Cref{Def:hcomp}) then $\Supp(f^*(A))=f^{-1}(\Supp(A))$.
\smallbreak
		\item\label{it:Supp-b} For any weak ring $B$ in $\cat S$, we have $\Supph(f_*(B)) = \fhomol(\Supph(B))$.
			Hence $\Supp(f_*(B)) = f(\Supp(B))$.
	\end{enumerate}
\end{Prop}

\begin{Rem}
	In particular, taking $B=\unit_{\cat S}$, we have the formula
	\begin{equation*}
		\Img(f) = \Supp(f_*(\unit_{\cat S}))
	\end{equation*}
	for the image of the map $f:\SpcS\to\SpcT$ induced by a geometric functor $f^*:\cat T \to \cat S$. This formula, first established in~\cite[Corollary~5.13]{Balmer20_bigsupport}, relies on the `big' categories $\cat T$ and $\cat S$. We want to extend this result to the setting where we are only given a tt-functor on the small objects. The resulting \cref{Thm:Img=Supp(A)} may be useful in other contexts.
\end{Rem}

\begin{Rec}
\label{Rec:brothers}%
	To prove the theorem, we first recall two twin lemmas concerning the existence of primes for an essentially small tt-category $\cat K$:
\begin{enumerate}[\rm(a)]
	\item\label{it:brother-1} The first one~\cite[Lemma~2.2]{Balmer05a} is purely tt-categorical. For every tt-ideal~$\cJ\subset\cK$ and every $\otimes$-multiplicative class of objects~$S\subset\cK$ such that $\cJ\cap S=\varnothing$ there exists a triangular prime~$\cP\in\Spc(\cK)$ such that $\cJ\subseteq\cP$ and~$\cP\cap S=\varnothing$.
	\item\label{it:brother-2} The second one~\cite[Lemma~3.8]{Balmer20_nilpotence} assumes that $\cK$ is rigid and involves the abelian tensor-category~$\mod\cK$. For every Serre $\otimes$-ideal~$\cat{I}\subset\mod\cK$ and every $\otimes$-multiplicative class of objects~$S\subset\cK$ such that $\cat{I}\cap\homol(S)=\varnothing$ there exists a homological prime~$\cB\in\Spch(\cK)$ such that $\cat{I}\subseteq\cB$ and~$\cB\cap \homol(S)=\varnothing$.
\end{enumerate}
\end{Rec}

\begin{Thm}
\label{Thm:Img=Supp(A)}%
	Let $\cT$ be a rigidly-compactly generated tt-category (\Cref{Rec:big}) and let $f^*\colon \Td\to \cL$ be a tt-functor to an essentially small tt-category~$\cL$. Let~$A\in \cT$ be a weak ring, not necessarily compact.
	\begin{enumerate}[\rm(a)]
	\item
	\label{it:Img=Supp(A)-1}%
		Suppose that every morphism $\alpha$ in~$\Td$ such that $f^*(\alpha)=0$ satisfies $A\otimes\alpha=0$. Then we have $\Supp(A)\subseteq\Img(\Spc(f^*))$.
	\smallbreak
	\item
	\label{it:Img=Supp(A)-2}%
		Suppose that every morphism $\alpha$ in~$\Td$ such that $A\otimes \alpha=0$ satisfies $f^*(\alpha)=0$. Then we have $\Img(\Spc(f^*))\subseteq\Supp(A)$.
	\end{enumerate}
\end{Thm}
\begin{proof}
	For~\eqref{it:Img=Supp(A)-1}, let $\cP\in\Supp(A)$, meaning that $\cP=\hcomp(\cB)$ for some maximal Serre $\otimes$-ideal~$\cB\subset\mod\Td$ such that $\bar{A}:=\homol_{\cB}(A)$ is non-zero in the residue category~$\bar{\cA}_{\cB}$. This implies that the unit $\bar{\unit}\into \bar{A}$ is a monomorphism by~\cite[Proposition~3.5]{Balmer20_bigsupport}. We need to show that $\cP\in\Img(\Spc(f^*))$ which amounts to proving that the tt-ideal $\ideal{f^*(\cP)}_{\cL}$ does not meet the $\otimes$-multiplicative class~$f^*(\Td\sminus\cP)$ in~$\cL$. (See \Cref{Rec:brothers}\,(a) for~$\cL$. Any prime~$\cQ\in\Spc(\cL)$ such that $\ideal{f^*(\cP)}_{\cL}\subseteq\cQ$ and $\cQ\cap f^*(\cK\sminus\cP)=\varnothing$ satisfies~$(f^*)\inv(\cQ)=\cP$.) If \textsl{ab absurdo} there was some $c\in\cP$ and some $d\in\Td\sminus\cP$ such that $f^*(d)\in\ideal{f^*(c)}_{\cL}$ then, by the usual arguments of~\cite{Balmer10b}, we would have $f^*(\xi_c\potimes{n}\otimes d)=0$ for some~$n\gg1$, where the map $\xi_c\colon b\to \unit$ is the homotopy fiber of~$\mathrm{coev}\colon\unit\to c\otimes c^\vee$ in~$\Td$. By hypothesis this implies that $A\otimes\xi_c\potimes{n}\otimes d=0$. Applying the tensor-functor~$\homol_{\cB}\colon\cT\to \bar{\cA}_{\cB}$ to this relation we see that $\homol_{\cB}(\xi_c)\potimes{n}$ is zero on the object~$\bar{A}\otimes\homol_{\cB}(d)$. However $\homol_{\cB}(c)=0$ since $c\in \cP=\hcomp(\cB)=\ker(\homol_{\cB})$ and therefore $\homol_{\cB}(\xi_c)$ is an isomorphism. We conclude that the object $\bar{A}\otimes\homol_{\cB}(d)$ must be zero, and therefore $\homol_{\cB}(d)=0$ since $\bar{\unit}\into\bar{A}$ is a monomorphism and~$\homol_{\cB}(d)$ is~$\otimes$-flat. So $d\in\ker(\homol_{\cB})=\cP$ which is a contradiction.

	For~\eqref{it:Img=Supp(A)-2}, let $\cP=(f^*)\inv(\cQ)$ for $\cQ\in\Spc(\cL)$. Consider the functor $g^*\colon \Td\to \cL\onto \cL/\cQ$. Note that $A\otimes\alpha=0$ also implies~$g^*(\alpha)=0$ and since $\cP=(g^*)\inv(0)$ we can replace~$f^*$ by~$g^*$ and assume for simplicity that $\cL$ is local ($0$ is prime) and that $\cP=\Ker(f^*)$. Consider now the commutative diagram (\Cref{Rec:mod})
	\begin{equation}
	\label{eq:aux-f^**}%
	\vcenter{\xymatrix@C=3em{
	\Td \ar[r]^-{f^*} \ar[d]_-{\homol}
	& \cL \ar[d]^-{\homol}
	\\
	\mod\Td \ar[r]^-{\tilde{f}^*}
	& \mod\cL.
	}}
	\end{equation}
	Let $\cat{I}\subset\mod\Td$ be the kernel of~$\tilde{f}^*$, that is, $\cat{I}=\SET{M\in\mod\Td}{\tilde{f}^*(M)=0}$ and let~$S\subset\Td$ be the $\otimes$-multiplicative complement of~$\cP$. By the above discussion, we have $\cat{I}\cap \homol(S)=\varnothing$. Hence (\Cref{Rec:brothers}\,(b)) there exists~$\cB\in\Spch(\Td)$ such that $\cat{I}\subseteq\cB$ and~$\cB\cap \homol(S)=\varnothing$. It follows from these properties that $\hcomp(\cB)=\homol\inv(\cB)=\cP$. So we get the result if we prove that $\homol_{\cB}(A)$ is non-zero in the residue category~$\bar{\cA}_{\cB}$, for then~$\cB$ provides the homological prime above~$\cP$ that detects~$A$. For this, let $I\in\Mod\Td$ be the kernel of~$\unit\to \homol(A)$ in~$\Mod\Td$. Since $I$ is a subobject of the finitely presented~$\unit$, we must have $I=\cup_{M\subseteq I}M$ where $M$ runs over the finitely presented subobjects of~$I$ (see~\cite[Lemma~3.9]{BalmerKrauseStevenson20} if necessary). Any such $M$ is the image of some~$\homol(\alpha)$ for some map~$\alpha\colon a\to \unit$ in~$\Td$. It follows that $A\otimes\alpha=0$ in~$\cT$ by using that this holds under restricted Yoneda and that $A$ is a weak ring. By hypothesis, this forces $f^*(\alpha)=0$. Therefore $\tilde{f}^*(M)=0$ in~$\mod\cL$ by exactness of~$\tilde{f}^*$ and commutativity of~\eqref{eq:aux-f^**}. Hence $M\in\cB$ and therefore $I=\cup_M M\in\Loc{\cB}$. It follows that $I\mapsto 0$ in~$\bar{\cA}_{\cB}$ and therefore $\bar{\unit}\into\homol_{\cB}(A)$ is a monomorphism.
\end{proof}

\section{Localization, completion and the Tate construction}
\label{sec:recollement}%

Let us remind the reader of finite recollements; see~\cite{Greenlees01}.

\begin{Hyp}
\label{Hyp}%
	Let $\cT$ be a `big' tt-category (\Cref{Rec:big}). Choose a Thomason subset~$Y$ of its spectrum and denote the complement of~$Y$ by~$U$:
	\begin{equation}
	\label{eq:YU}%
	Y\subseteq\SpcT
	\qquadtext{and}
	U\coloneqq Y^\complement=\SpcT\sminus Y.
	\end{equation}
	The choice of~$Y$ is the only parameter in our entire discussion. It amounts to choosing the corresponding tt-ideal~$\TYc\coloneqq \SET{ c \in \Tc=\Td}{\supp(c)\subseteq Y}$ of small objects supported on~$Y$. (We explain in \Cref{Rem:TYc} why we write $\TYc$ instead of~$\Td_Y$.) The latter generates the localizing ideal $\TY\coloneqq\Loc{\TYc}=\Loco{\TYc}$ of~$\cT$. Brown--Neeman Representability provides two semi-orthogonal decompositions\,(\footnote{\,Recall that $\cT=\langle\cA\perp\cB\rangle$, for two triangulated subcategories~$\cA,\cB\subseteq\cT$, means that every object $t \in \cT$ fits in an exact triangle~$a\to t\to b\to \Sigma a$ with $a\in\cA$ and~$b\in\cB$ and that there is no non-zero morphism in~$\cT$ from objects of~$\cA$ to those of~$\cB$. In that case, the inclusion $\cA\into \cT$ admits a right adjoint identifying~$\cA$ with~$\cT/\cB$ and $\cB\into \cT$ admits a left adjoint identifying~$\cB$ with~$\cT/\cA$.})
	\begin{equation}
	\label{eq:decompositions}%
	\cT=\langle \TY\perp \TYp\rangle
	\qquadtext{and}
	\cT=\langle \TYp\perp \TYpp\rangle.
	\end{equation}
	In the literature, $\TY$ is sometimes called the subcategory of~\emph{$Y$-torsion objects}; objects of its right-orthogonal $\TYp\coloneqq\SET{t\in\cT}{\Homcat{T}(s,t)=0\textrm{ for all }s\in\TY}=(\TYc)^\perp$ are called \emph{local}; and finally the double orthogonal $\TYpp=(\TYp)^\perp$ is usually called the subcategory of \emph{$Y$-complete} objects. When the tt-ideal $\TYc$ is given a name, say~$\cJ$, some authors speak of `$\cJ$-torsion', `$\cJ$-local' and~`$\cJ$-complete'.

	Since the local~$\TYp$ appears in both decompositions~\eqref{eq:decompositions}, the quotient~$\cT/\TYp$ has two equivalent realizations, namely the $Y$-torsion and the $Y$-complete subcategories. Hence they are equivalent via the composite $\TY \into \cT\onto \TYpp$ of the inclusion of~$\TY$ and the left adjoint to the inclusion of~$\TYpp$. We have two \emph{equivalent} recollements
	\begin{equation}
	\label{eq:recollements}%
	\vcenter{\xymatrix@C=4em{
	\TY \ar@{ >->}@<-1.5em>[d]_-{\mathbf{incl}} \ar@{ >->}@<1.5em>[d]^-{[\eY,-]}
	& \cong
	& \TYpp \ar@{ >->}@<-1.5em>[d]_-{\eY\otimes-} \ar@{ >->}@<1.5em>[d]^-{\mathbf{incl}}
	\\
	\cT \ar@{->>}[u]|-{\eY\otimes-\vbuff} \ar@{->>}@<-1.5em>[d]_-{\fY\otimes-} \ar@{->>}@<1.5em>[d]^-{[\fY,-]}
	& =
	& \cT \ar@{->>}[u]|-{[\eY,-]} \ar@{->>}@<-1.5em>[d]_-{\fY\otimes-} \ar@{->>}@<1.5em>[d]^-{[\fY,-]}
	\\
	\TYp \ar@{ >->}[u]|-{\mathbf{incl}\vbuff}
	& =
	& \TYp \ar@{ >->}[u]|-{\mathbf{incl}\vbuff}
	}}
	\end{equation}
	with identical `bottom' part. We highlight the inclusions~`$\incl$' of subcategories. The other functors are uniquely defined as adjoints. We indicate in \eqref{eq:recollements} how to compute those adjoints via the tensor structure. (It follows that the equivalence $\TY\isoto \TYpp$ is given by~$[\eY,-]$ with inverse~$\eY\otimes-$.) This involves the exact triangle
	\begin{equation}
	\label{eq:idemp-triangle}%
	\eY \to \unit\to \fY\to \Sigma\eY
	\end{equation}
	in~$\cT$ that is uniquely characterized by the properties $\eY\in \TY$ and $\fY\in \TYp$; this triangle comes from the semi-orthogonal decomposition $\cT=\langle\TY\perp\TYp\rangle$ applied to the object~$t=\unit$. See details in~\cite{BalmerFavi11} where~\eqref{eq:idemp-triangle} is called an \emph{idempotent triangle} since~$\eY\potimes{2}\cong\eY$ and~$\fY\potimes{2}\cong\fY$. For an object~$t\in\cT$ its triangles with respect to the decompositions~\eqref{eq:decompositions} can be extracted from~\eqref{eq:idemp-triangle} by applying~$-\otimes t$ and~$[-,t]$:
	\begin{align}
	\label{eq:triangle-t-loc}%
	\textrm{For }\cT & =\langle \TY\perp \TYp\rangle:
	\qquad
	\kern.3em
	\eY\otimes t\to t \to \fY\otimes t \to \Sigma \eY\otimes t
	\kern.9em
	\\
	\label{eq:triangle-t-coloc}%
	\textrm{For }\cT & =\langle \TYp\perp \TYpp\rangle:
	\qquad
	[\fY,t]\to t \to [\eY,t] \to \Sigma [\fY,t].
	\end{align}
	The idempotent endofunctors~$\eY\otimes-,\ \fY\otimes-,\ [\eY,-],\ [\fY,-]\colon \cT\to \cT$ satisfy:
	\begin{align}
	\label{eq:TY}%
	\TY & =\Img(\eY\otimes-)=\Ker(\fY\otimes-)
	\\
	\label{eq:TYp}%
	\TYp & =\Img(\fY\otimes-)=\Img([\fY,-])=\Ker(\eY\otimes-)=\Ker([\eY,-])
	\\
	\label{eq:TYpp}%
	\TYpp & =\Img([\eY,-])=\Ker([\fY,-]).
	\end{align}
\end{Hyp}

\begin{Rem}
\label{Rem:TU}%
	By the Neeman--Thomason Theorem, the localization $\fY\otimes-\colon\cT\onto\TYp$ is a geometric functor whose image $\TYp\cong\cT/\TY$ is a genuine `big' tt-category in the sense of \Cref{sec:basics}. In tt-geometry, it is usually denoted
	\begin{equation}
		\TU:=\cT/\TY\cong\TYp.
	\end{equation}
	Its subcategory of small objects is given by~$\TUd=(\Td/\TYc)^\natural$, the idempotent-completion of the corresponding quotient of small objects. See~\cite{Neeman92b}. It follows that its spectrum is just~$U$, embedded in~$\SpcT$ via the injective map induced by localization: $\Spc(\TUd)\cong U$. We think of~$\TU$ as the restriction of~$\cT$ to~$U$. For instance, in algebraic geometry, if~$\cT$ is the derived category of a quasi-compact and quasi-separated scheme~$X$ and~$U\subseteq X\cong\SpcT$ is a quasi-compact open subset then~$\TU$ recovers the derived category of~$U$ and~$\cT\onto \TU$ is the usual restriction to~$U$. In view of this, the objects of~$\TYp$ should be called \emph{$U$-local} (not $Y$-local).
\end{Rem}

The story is more complicated with the $Y$-torsion and \mbox{$Y$-complete} subcategories.
\begin{Rem}
\label{Rem:TY-TYpp}%
	At first sight, both $\TY$ and~$\TYpp$ have many desirable properties. They are tensor-triangulated categories with coproducts. The tensor in~$\TY$ is the one in~$\cT$ but with unit~$\eY$; the tensor in~$\TYpp$ is $s\hat\otimes t\coloneqq[\eY,s\otimes t]$, the tensor in~$\cT$ followed by~$[\eY,-]$; it has unit~$[\eY,\unit]\cong[\eY,\eY]$. The upward functors in~\eqref{eq:recollements} that go from~$\cT$ to these categories~$\TY$ and~$\TYpp$ are Bousfield localizations with respect to tensor-ideals. These localizations preserve the tensor (hence dualizable objects) and they have adjoints on both sides. Furthermore, the categories $\TY$ and~$\TYpp$ are compactly generated by dualizable objects, namely by $\TYc$ which is the subcategory of compact objects in both of them:
	\begin{equation}
	\label{eq:TYc}%
		(\TY)\cc=(\TYpp)\cc=\TYc.
	\end{equation}
	For~$\TY=\Loc{\TYc}$, this follows from~\cite[Lemma~2.2]{Neeman92b}: the compact objects in a compactly generated subcategory are given by the thick closure of the generators. Transporting this under the equivalence~$[\eY,-]\colon \TY\smash{\isoto} \TYpp$ and using that $[\eY,c]\cong [\eY\otimes c^\vee,\unit]\cong[c^\vee,\unit]\cong [\unit,c]\cong c$ when $c\in\TYc$ is dualizable and supported on~$Y$, we get $(\TYpp)\cc=\TYc$ again.
\end{Rem}

\begin{Rem}
\label{Rem:split}%
	Unfortunately, neither $\TY$ nor~$\TYpp$ is rigidly-compactly generated in general. Indeed, their tensor-unit is only compact in the split case where~$\cT\cong\cT_1\times\cT_2$ and $Y=\Spc(\cT_1\dd)$ and $U=\Spc(\cT_2\dd)$. See~\cite[Remark~2.12]{BalmerSanders26}. Thus the localizations $\cT\onto \TY$ and~$\cT\onto \TYpp$ are \emph{not} geometric functors of `big' tt-categories in the sense of \Cref{Rec:geometric} except in the split case.
\end{Rem}

\begin{Rem}
\label{Rem:comp-def}%
	Despite its snazzy name, the $Y$-complete subcategory $\TYpp$ is not any better than the `naive' $Y$-torsion subcategory~$\TY$. There is absolutely no more mathematical information in the right-hand side of~\eqref{eq:recollements} than in its tt-equivalent left-hand side.

	We must nevertheless acknowledge the weight of tradition and the vast literature that presents $\TYpp$ as \emph{the} completion of~$\cT$ along~$Y$. Mathematicians like to compose functors with their right adjoint, as in the case of extension-of-scalars $B\otimes_A-\colon A\textrm{-Mod}\to B\textrm{-Mod}$: We like to think of~$B$ as an $A$-module and we like to think of~$B\otimes_A-$ as a monad on $A\textrm{-Mod}$. If we compose the localization $\cT\onto\TYpp$ (or $\cT\onto\TY$) in~\eqref{eq:recollements} with its right adjoint we get (in \emph{both} cases!) the functor
	\begin{equation}
		[\eY,-]\colon \cT\to \cT.
	\end{equation}
	This is just a Bousfield localization. This functor \emph{does} have a `completion' vibe to~it, as we discuss in~\cite[Remark~2.6]{BalmerSanders26}. For instance, the following holds true.

	In the case of the derived category $\cT=\Der(R)$ of a commutative noetherian ring~$R$, the above functor $[\eY,-]$ almost agrees with $\hat{R}_I\otimes-\colon \Der(R)\to \Der(R)$ where~$\hat{R}_I$ is the $I$-adic completion and $Y=V(I)$. The agreement $[\eY,c]\cong\hat{R}_I\otimes c$ holds on \emph{perfect} complexes~$c$. However it fails for `big' objects, like the object~$\hat{R}_I$ itself, since $[\eY,\hat{R}_I]$ is~$\hat{R}_I$ again whereas $\hat{R}_I\otimes\hat{R}_I$ is not. The (Bousfield) completion~$\TYpp$ cannot recover~$\Der(\hat{R}_I)$, since the latter is always a legitimate `big' tt-category. Nobody expects extension-of-scalars along~$\smash{R\to \hat{R}_I}$ to be a mere localization.

	In~\cite{BalmerSanders26}, we prove that something about $\Der(\hat{R}_I)$ can nevertheless be read off~$\TYpp$ in the above example of~$\cT=\Der(R)$. Namely, the compact-dualizable objects~$\Dperf(\hat{R}_I)$ in~$\Der(\hat{R}_I)$ can be recovered as the \emph{dualizable} objects in~$\Der(R)_{Y}^{\perp\perp}$.

	This result suggests a better notion of completion of~$\cT$ along~$Y$ in general; namely it should be a `big' tt-category~$\hat{\cT}$ whose subcategory of small objects is equal to
	\begin{equation}
	\label{eq:TYd}%
	\hTd\coloneqq(\TYpp)\dd
	\end{equation}
	the subcategory of dualizable objects in the $Y$-complete~$\TYpp$, or equivalently~(!) the dualizable objects in~$\TY$. In both cases, we mean the dualizable objects in the tt-category itself, not in the ambient category~$\cT$ under any of the embeddings. This is one possible definition of the `big'~$\hat{\cT}$ adopted by our $\infty$-friends Naumann--Pol--Ramzi~\cite{NaumannPolRamzi24}. In the presence of an $\infty$-categorical model, they construct $\hat{\cT}$  as the Ind-completion of the~$\hTd$ prescribed in~\eqref{eq:TYd}; see~\cite[Definition~5.1]{NaumannPolRamzi24}.
\end{Rem}

\begin{Rem}
\label{Rem:TYc}%
With~\eqref{eq:TYd}, we began considering dualizable objects in~$\TYpp\cong\TY$. This is the reason why ever since \Cref{Hyp} we have been systematically writing $\TYc$ instead of~$\Td_Y$ for the tt-ideal $(\Tc)_Y=(\Td)_Y$ of the compact-dualizables in~$\cT$ supported on~$Y$. We want to avoid any possible confusion with $(\TY)\dd$.
\end{Rem}

\begin{Rem}
\label{Rem:exotic}%
	We could make other choices for the completion of~$\Td$ along~$Y$ instead of the dualizables~$\hTd=(\TYpp)\dd$ in the $Y$-complete subcategory chosen in~\eqref{eq:TYd}. The smallest choice is the thick envelope $\smash{\hTd_0}=\smash{\Thick_{\TYpp}}([\eY,\Td])$ of the image of~$\Td\to \hTd$. This is the same as $\hTd$ when the latter is generated by the unit, as in the ring case discussed above, but in general we should expect `exotic' dualizable objects in~$\TYpp$, that is, outside of~$\hTd_0$. Exotic dualizables exist in homotopy theory, for instance in the $K(n)$-local category; see \cite[Section 5]{BensonIyengarKrausePevtsova23} with details in \cite[Section 15.1]{HoveyStrickland99}. In contrast, \cite{BensonIyengarKrausePevtsova24pp} establishes that there are no exotic dualizables for the stable module category. Compare also~\cite{NaumannPol24}.

	For any such choice of an intermediate tt-subcategory~$\cL$, between $\hTd_0$ and~$\hTd$, the fully faithful functors $\hTd_0\hook \cL\hook \hTd$ induce surjective maps on spectra $\Spc(\hTd) \onto \Spc(\cL)\onto \Spc(\hTd_0)$ by~\cite{Balmer18}. In particular, all those spectra have the same image in the original~$\SpcT$. One can also prove that these maps $\Spc(\hTd) \onto \Spc(\cL)\onto \Spc(\hTd_0)$ are quotient maps with connected fibers, by upcoming work of the second author~\cite{Sanders25bpp}.
\end{Rem}

Let us summarize our discussion.
\begin{Def}
\label{Def:completion}%
	We call $\TYpp$ the \emph{Bousfield completion} of~$\cT$ along~$Y$ and call~$\hTd=(\TYpp)\dd$ as in~\eqref{eq:TYd} \emph{the completion of~$\Td$ along~$Y$}. When we need to emphasize $Y$, we shall write~$\smash{\hTuYd}$ instead of~$\smash{\hTd}$ but for most of the paper $Y$ is clear from context. We denote the \emph{completion} tt-functor $[\eY,-]\colon \cT\to \TYpp$ on dualizables by
	\begin{equation}\label{eq:completion}%
	\vcenter{\xymatrix@R=0em{
	\widehat{(-)}: \kern1em \Td \ar[r] & \kern.8em \hTd \kern.8em
	\\
	\kern4em c \kern.5em \ar@{|->}[r] & \hat{c}=[\eY,c].}}
	\end{equation}
	By extension, we also denote by~$\hat{c}=[\eY,c]$ the object~$\hat{c}$ viewed in~$\cT$. For instance, $\hat{\unit}=[\eY,\unit]$ is the completed unit. We write $\varphi_Y$ for the induced map on spectra
	\begin{equation}\label{eq:varphiY}
		\varphi_Y\colon \Spc(\hTd)\to \SpcT.
	\end{equation}
	Again, we drop the mention of~$Y$ and simply write~$\varphi$ when clear from context.
\end{Def}
\begin{Rem}
	At first, some readers might dislike the absence of~$Y$ in the above notation~$\hat{\cT}$ for completion. This simplification follows standard practice when completing a topological space: One omits the metric, unless several metrics are involved. In our case, there is another reason. We shall see in \cref{Thm:Spc(compl)-on-Y} that~$Y$ identifies with a Thomason subset of~$\Spc(\hTd)$; therefore $(\hTd)_Y$ already has a meaning --- the objects of~$\hTd$ supported on~$Y$ --- which would collide with writing the completion as ${\hat{\cT}}_{Y}^{\,\dname}$, for instance. Mathematicians are overusing indices.
\end{Rem}

We can generalize the example considered in \Cref{Rem:comp-def}.
\begin{Exa}
\label{Exa:hatR_I}%
	Let $R$ be a commutative ring, let~$I=\langle s_1,\ldots,s_r\rangle\subseteq R$ be a finitely generated ideal, and let $\hat{R}_I=\lim_n R/I^n$ be the $I$-adic completion. Suppose that $\underline{s}=(s_1,\ldots,s_r)$ is `Koszul-complete', meaning that the canonical map on Koszul complexes $\Kos_R(\underline{s})\to \Kos_{\hat{R}_I}(\underline{s})$ is a quasi-isomorphism. If $R$ is noetherian then every~$\underline{s}$ is Koszul-complete by~\cite[Proposition~3.17]{BalmerSanders26}. Let $\cT\coloneqq\Der(R)$ and~$Y\coloneqq V(I)\subseteq\Spec(R)\cong\SpcT$. By~\cite[Theorem~5.1]{BalmerSanders26}, we have $\hat\unit\cong\hat{R}_I$ and a tt-equivalence $\hTd\cong\Dperf(\hat{R}_I)$ that turns the tt-functor $\widehat{(-)}\colon \Td\to \hTd$ into extension-of-scalars~$\Dperf(R)\to \Dperf(\hat{R}_I)$. The map $\varphi_Y\colon \Spc(\hTd)\to \SpcT$ of~\eqref{eq:varphiY} identifies with the map $\Spec(\hat{R}_I) \to \Spec(R)$ induced by $I$-adic completion on ordinary Zariski spectra.
\end{Exa}

So the completion $\hTd$ is what we expect it to be in this example. In particular it is essentially small. It might not be clear \textsl{a priori} why $\hTd$ is so in general.
\begin{Prop}
\label{Prop:dualizables}%
	Let $\cS$ be a tt-category that is compactly generated as in~\cite{HoveyPalmieriStrickland97}.
\begin{enumerate}[\rm(a)]
	\item
		The tensor of a compact object with a dualizable object is compact: $\cS\dd\otimes\cS\cc\subseteq\cS\cc$.
	\item
		The subcategory of dualizable objects $\cS\dd$ is essentially small.
\end{enumerate}
\end{Prop}

\begin{proof}
	For~(a), if~$d\in\cS\dd$ is dualizable and~$c\in\cS\cc$ is compact, we have $\Hom(c\otimes d,-)\cong\Hom(c,d^\vee\otimes-)$ and both $d^\vee\otimes-$ and~$\Hom(c,-)$ commute with coproducts.

	For~(b), for every infinite cardinal~$\alpha$, let us denote by~$\cL_{\alpha}$ the $\alpha$-localizing subcategory generated by the compact objects, that is, closing~$\Sc$ under triangles and coproducts of fewer than~$\alpha$ objects. It follows from~(a) and cocontinuity of the tensor that $\cS\dd\otimes\cL_{\alpha}\subseteq \cL_{\alpha}$. One can check that each $\cL_{\alpha}$ is essentially small; for instance, it is contained in the essentially small category of $\alpha$-compact objects in the sense of Neeman~\cite{Neeman01}. Since $\cS = \Loc{\Sc}$, we have $\cS=\cup_{\alpha}\cL_\alpha$ and therefore there exists~$\alpha$ large enough so that $\unit\in\cL_\alpha$. In that case, $\cS\dd=\cS\dd\otimes\unit\subseteq\cL_{\alpha}$ is essentially small.
\end{proof}
\begin{Rem}
	One can adapt the above argument to establish that $\Sd$ is essentially small for tt-categories $\cS$ that are only `well generated'.
\end{Rem}

\begin{Cor}
The category $\hTd$ of dualizable objects in~$\TYpp$ is essentially small.
\end{Cor}

\begin{proof}
	The Bousfield completion $\TYpp$ is compactly generated (\cref{Rem:TY-TYpp}).
\end{proof}

Besides commutative algebra, completion also appears in representation theory.
\begin{Exa}
\label{Exa:KInj}%
	Let $k$ be a field of characteristic~$p>0$ and let~$G$ be a finite group. The tensor of $k$-linear $G$-representations is the tensor over~$k$ with diagonal \mbox{$G$-action.} Let $\cT=\KInj(kG)$ be the homotopy category of complexes of injective~$kG$-modules and consider the localization~$Q\colon\cT\onto\Der(kG)$ onto the derived category, which mods out the subcategory $\Kac\Inj(kG)$ of acyclic complexes of injectives. Write $J\colon \Kac\Inj(kG)\into \KInj(kG)$ for the inclusion. They fit into a recollement~\cite{BensonKrause08}
	\begin{equation}
	\label{eq:KInj-rec}%
	\vcenter{\xymatrix{
	\Der(kG) \ar@{ >->}@<-1em>[d]_-{Q_{\lambda}} \ar@{ >->}@<1em>[d]^-{Q_{\rho}}
	\\
	\KInj(kG) \ar@{->>}[u]|-{Q\vbuff} \ar@{->>}@<-1em>[d]_(.45){J_{\lambda}} \ar@{->>}@<1em>[d]^-{J_{\rho}}
	\\
	\Kac\Inj(kG)\ar@{ >->}[u]|(.55){J \vbuff}
	}}
	\kern5em
	\vcenter{\xymatrix{
	\Dperf(kG) \ar@{ >->}[d]
	\\
	\Db(kG) \ar@{->>}[d]
	\\
	\stab(kG)
	}}
	\end{equation}
	whose compact objects are displayed on the right-hand side. More precisely, $\Db(kG)$ means $\Db(kG\mmod)$ which embeds in~$\KInj(kG)$ via~$Q_\rho$. Using that~$Q_{\lambda}$ agrees with~$Q_{\rho}$ on perfect complexes, the compacts of~$\Kac\Inj(kG)$ identify with the stable module category $\stab(kG)=kG\mmod/kG\textrm{-proj}$. Consequently, $\Kac\Inj(kG)\cong\Stab(kG)$ is equivalent to the big stable module category.

	In this example, $\Dperf(kG)=\Td_{Y}$ where $Y=\{\ast\}$ is the closed point in the homogenous spectrum of cohomology~$\smash{\Spech(\rmH^\sbull(G,k))}\cong\SpcT$, the latter being the tt-spectrum of $\Db(kG)$. Therefore the Bousfield completion of~$\cT=\KInj(kG)$ along~$Y=\{\ast\}$ is the derived category~$\Der(kG)$. The `big' completion would recover~$\KInj(kG)$ again, as the dualizables~$\smash{\hTd}$ in~$\Der(kG)$ are just $\smash{\Db(kG)}$, which are indeed the compacts of~$\KInj(kG)$. In other words, the local tt-category $\cT=\KInj(kG)$ should be considered `complete' (at its closed point).

	Note that in this example the ring object~$\hat\unit_{\cT}=[\eY,\unit]=Q_\rho(\unit)$ is~$\unit_{\cT}$, confirming the fact that $\cT$ is $Y$-complete. Its localization~$\fY\otimes[\eY,\unit]=\fY=\unit_{\Stab(kG)}$ is the unit in the stable module category, whose graded endomorphism ring is Tate cohomology. This fact explains the following terminology.
\end{Exa}

\begin{Def}
\label{Def:Tate}%
	Under \Cref{Hyp}, the \emph{Tate functor} $\tY : \cat T \to \cat T$ is defined by
	\[
		\tY(t) \coloneqq \fY \otimes [\eY,t]\cong [\fY,\Sigma \eY\otimes t].
	\]
	The above natural isomorphism is the `Warwick Duality' of Greenlees~\cite[Corollary~2.5]{Greenlees01}. One observes that $\tY$ is lax monoidal. Thus $\tY(t)$ is a (weak) ring whenever $t$ is a (weak) ring. In particular, we have the \emph{Tate ring}
	\[
		\tY(\unit)=\fY\otimes [\eY,\unit]\cong[\fY,\Sigma\eY]
	\]
	which is a commutative ring object in~$\cT$. (The Tate ring~$\tY(\unit)$ was denoted $\tY$ in the introduction, for simplicity.) Note that every~$\tY(t)$, and in particular the Tate ring~$\tY(\unit)$, is a $U$-local object of~$\cT$ by construction.
\end{Def}

The vanishing of the Tate functor forces a direct-sum decomposition.
\begin{Prop}
\label{Prop:Tate-decomposition}%
	Let $t\in \cT$ be an object such that $\tY(t)=0$. Then the exact triangle~\eqref{eq:triangle-t-loc} splits and we have $t\cong(\eY\otimes t)\oplus(\fY\otimes t)$.
\end{Prop}
\begin{proof}
	We have $\cT(\fY\otimes t,\Sigma\eY\otimes t)\cong\cT(t,[\fY,\Sigma\eY\otimes t])\cong\cT(t,\tY(t))$. Hence $\tY(t)=0$ forces the third morphism in~\eqref{eq:triangle-t-loc} to be zero.
\end{proof}

\begin{Rem}
	For any object $t\in \cat T$, tensoring the idempotent triangle \eqref{eq:idemp-triangle} by $[\eY,t]$ and using that $\eY \otimes [\eY,t] \simeq \eY\otimes t$, we obtain an exact triangle
	\begin{equation}\label{eq:norm-exact-triangle}
		\eY\otimes t \to [\eY,t] \to \tY(t) \to \Sigma \eY \otimes t
	\end{equation}
	called the `norm' exact triangle. The name comes from the `norm cofiber sequence' in equivariant homotopy theory; see \cref{rem:equivariant-norm} below.
\end{Rem}

We end this section with an easy observation that will come in handy later.
\begin{Lem}
\label{Lem:YW=0}%
	Let $Y,W\subseteq\SpcT$ be disjoint Thomason subsets: $Y\cap W=\varnothing$. Then $[s,t]=0$ in~$\cT$ for all~$s\in\TY$ and~$t\in \cT_W$.
\end{Lem}
\begin{proof}
	This is clear if~$s\in\TYc$ and~$t\in\Td_W$ are dualizable since $\supp([s,t])=\supp(s^\vee)\cap\supp(t)=\varnothing$. It passes to~$t\in\cT_W$ arbitrary by compactness of~$s\in\TYc$. It then passes to every~$s\in\TY$ since $\Ker([-,t])$ is a localizing subcategory of~$\cT$.
\end{proof}

\section{The spectral map induced by completion}
\label{sec:Spc-completion}%

We next discuss the map $\varphi\colon \Spc(\hTd)\to \SpcT$ induced by the tt-functor $\widehat{(-)}=[\eY,-]\colon\Td\to \hTd$ of completion along~$Y$. We keep the notation of \cref{Hyp} and \cref{Def:completion}.

\begin{Prop}
\label{Prop:pseudo-proj-formula}%
	Consider the commutative ring object~$[\eY,\unit]$ in~$\cT$. For a morphism~$\alpha$ in~$\Td$ we have $\hat{\alpha}=0$ in~$\hTd$ if and only if~$[\eY,\unit]\otimes \alpha=0$ in~$\cT$.
\end{Prop}
\begin{proof}
	We have in~$\cT$ that $[\eY,c]\cong[\eY,\unit]\otimes c$ for every dualizable~$c\in\Td$.
\end{proof}

Recall our discussion of the support of big (ring) objects in \Cref{sec:basics}.
\begin{Cor}\label{cor:img-of-compl}
	The image of the continuous map $\Spc(\hTd)\to \SpcT$ induced by completion is~$\Supp([\eY,\unit])$.
\end{Cor}
\begin{proof}
	Apply \Cref{Thm:Img=Supp(A)} to the ring~$A=[\eY,\unit]$ of \Cref{Prop:pseudo-proj-formula}.
\end{proof}

We now want to show that $\varphi\colon\Spc(\hTd)\to \SpcT$ is a homeomorphism above~$Y$.
\begin{Rem}
\label{Rem:comp-Y-prep}%
	As explained in \Cref{Rem:TY-TYpp}, the compact objects in~$\TYpp$ are given by~$\TYc=(\Td)_Y$, on which completion~$[\eY,c]\cong c$ is isomorphic to the identity. (Recall \Cref{Rem:TYc}.) These compact objects are dualizable in~$\TYpp$ since $[\eY,-]\colon \cT\to \TYpp$ is a tensor functor. Hence we have a rigid tt-category~$\hTd=(\TYpp)\dd$ that contains the subcategory~$(\TYpp)\cc$ of compact objects and the latter is equivalent (equal) to the tt-ideal~$\TYc$ in the original category.

	By \Cref{Prop:dualizables}\,(a), the class of compact objects~$\TYc$ forms a tt-ideal in~$\hTd$. For every $c\in\TYc$ and every $d\in\hTd$, we can therefore speak of the properties of the object~$c\otimes d\in\TYc\subseteq\Td$ as a small object of the original category~$\cT$. We do so in the preparatory lemmas below.
\end{Rem}

\begin{Lem}
\label{Lem:A-hat}%
	Let $\cA$ be a class of objects in~$\TYc$ and the corresponding class~$\hat{\cA}=\SET{\hat{a}}{a\in\cA}$ in~$\hTd$. Then $\cA$ is a tt-ideal in~$\Td$ if and only if~$\hat\cA$ is a tt-ideal in~$\hTd$.
\end{Lem}
\begin{proof}
	Since $\TYc\cong(\TYpp)\cc$ is an equivalence of triangulated categories, $\cA$ is a thick triangulated subcategory of~$\TYc$ if and only if~$\hat{\cA}$ is so in~$\hTd$. For the `ideal' property, since completion $\Td\to \hTd$ is a tt-functor, it is clear that $\cA$ is a tt-ideal in~$\Td$ \emph{if} $\hat{\cA}$ is one in~$\hTd$. Conversely, suppose that $\cA$ is a tt-ideal in~$\TYc$. Let $a\in\cA$ and~$d\in\hTd$ and let us show that $\hat{a}\otimes d$ belongs to~$\hat{\cA}$. Note that $\hat{a}\otimes d$ is compact hence $\hat{a}\otimes d\simeq\hat{c}$ for some~$c\in\TYc$. Therefore, $\hat{a}\potimes{2}\otimes \hat{a}^\vee\otimes d\simeq\widehat{a \otimes a^\vee\otimes c}$ belongs to~$\hat{\cA}$ since~$\cA$ is an ideal. On the other hand, $\hat{a}$ is a direct summand of $\hat{a}\potimes{2}\otimes\hat{a}^\vee$ in~$\hTd$ by dualizability and it follows that $\hat{a}\otimes d$ is a direct summand of~$\hat{a}\potimes{2}\otimes \hat{a}^\vee\otimes d$, hence remains in~$\hat{\cA}$.
\end{proof}

\begin{Lem}
\label{Lem:comp-Y}%
	Let $c\in\TYc$ be a small object in the original category supported on~$Y$.
\begin{enumerate}[\rm(a)]
\item
\label{it:comp-Y-a}%
	Let $\cP\in\Spc(\Td)$ be a prime in the support of~$c$. Consider the subcategory of~$\hTd$
	\begin{equation*}
	\label{eq:comp-Y}%
		\fhat{\cP}:=\SET{d\in\hTd}{\exists\,a\in\TYc\cap \cP\textrm{ such that }\hat{a}\simeq\hat{c}\otimes d}.
	\end{equation*}
	Then~$\smash{\fhat{\cP}}$ is a prime tt-ideal of~$\hTd$ that belongs to the support of~$\hat{c}$. Moreover, its image under~$\varphi\colon\Spc(\hTd)\to\SpcT$ is~$\cP$.
\smallbreak
\item
\label{it:comp-Y-b}%
	Let $\cQ\in\Spc(\hTd)$ be a prime in the support of~$\hat{c}$ and let~$\cP=\SET{b\in\Td}{\hat{b}\in\cQ}$ be its image under~$\varphi\colon\Spc(\hTd)\to\SpcT$. Then $\cP$ belongs to the support of~$c$ and, with the notation of~\eqref{it:comp-Y-a}, we have $\fhat{\cP}=\cQ$.
\end{enumerate}
\end{Lem}

\begin{proof}
	Consider the tt-ideal~$\cA=\TYc\cap \cP$ of~$\Td$ and the corresponding tt-ideal~$\hat{\cA}=\SET{\hat{a}}{a\in\cA}$ of~$\hTd$, applying~\Cref{Rem:comp-Y-prep,Lem:A-hat}.

	In~\eqref{it:comp-Y-a}, by construction, $\fhat{\cP}$ consists of all dualizable objects~$d\in\hTd$ such that $\hat{c}\otimes d$ belongs to~$\hat{\cA}$. This forms a tt-ideal. Let us check that it is prime. Suppose that $d,d'\in\hTd$ are such that $d\otimes d'\in\fhat{\cP}$. Consider the compact objects~$\hat{c}\otimes d$ and~$\hat{c}\otimes d'$. By \Cref{Rem:comp-Y-prep}, there exists $a,a'\in\TYc$ such that $\hat{c}\otimes d\simeq\hat{a}$ and $\hat{c}\otimes d'\simeq\hat{a}'$ and our assumption that $\hat{c}\otimes d\otimes d'\in\hat{\cA}$ implies that $\hat{a}\otimes\hat{a}'\simeq\hat{c}\otimes d\otimes \hat{c}\otimes d'$ also belongs to~$\hat{\cA}$. In other words, $a\otimes a'\in\TYc\cap \cP$ and in particular $a\otimes a'$ belongs to the prime~$\cP$. This forces one of~$a$ or~$a'$ to belong to~$\cP$, and in turn~$d\in\fhat{\cP}$ or~$d'\in\fhat{\cP}$ respectively.

	Furthermore, $\fhat{\cP}$ is proper for otherwise $\hat{c}\in\hat{\cA}$ which forces~$c\in\cA\subseteq \cP$ and this contradicts the hypothesis that $\cP$ belongs to the support of~$c$. Similarly, $\hat{c}\in\fhat{\cP}$ would force $c\potimes{2}\in\cP$, leading to the same contradiction.

	It is clear that $\SET{\hat{b}}{b\in\cP}\subseteq\fhat{\cP}$. Conversely, if~$b\in\Td$ satisfies~$\hat{b}\in\fhat{\cP}$ then $c\otimes b\in\cA\subseteq \cP$ and, using again that $c\notin\cP$ and that $\cP$ is prime, we see that $b\in\cP$. In short, $\varphi(\fhat{\cP})=\SET{b\in\Td}{\hat{b}\in\fhat{\cP}}$ is equal to~$\cP$ as claimed.

	Let us now turn to~\eqref{it:comp-Y-b}. Since $\cP$ is the preimage of~$\cQ$ under completion, $\hat{c}\notin Q$ forces $c\notin\cP$. Let us prove that $\fhat{\cP}=\cQ$. Let $d\in\fhat{\cP}$. This means that $\hat{c}\otimes d\simeq\hat{a}$ for some~$a\in\TYc\cap \cP$. From $a\in\cP=\varphi(\cQ)$ we have $\hat{a}\in\cQ$. Hence we have $\hat{c}\otimes d\simeq\hat{a}\in\cQ$ and since by assumption~$\hat{c}\notin\cQ$ we conclude that $d\in\cQ$. This proves~$\fhat{\cP}\subseteq\cQ$. Conversely, let $d\in\cQ$ and consider $a\in\TYc$ such that $\hat{c}\otimes d\simeq\hat{a}$ by \Cref{Rem:comp-Y-prep} again. We have $\hat{a}\in\cQ$ hence $a\in\cP$ by the definition of~$\cP$. Therefore, $a\in\TYc\cap \cP$ and $\hat{c}\otimes d\simeq\hat{a}$ exactly means that $d\in\fhat{\cP}$ by~\eqref{it:comp-Y-a}. In short $\cQ\subseteq\fhat{\cP}$ and we are done.
\end{proof}

\begin{Thm}
\label{Thm:Spc(compl)-on-Y}%
Let $\varphi\colon\Spc(\hTd) \to \Spc(\Td)$ be the continuous map induced by completion along the Thomason subset~$Y$.
\begin{enumerate}[\rm(a)]
\item
\label{it:Spc-Y-a}%
The restriction $\varphi\inv(Y)\to Y$ is a bijection. If $Y$ is Thomason closed then $\varphi\inv(Y)\to Y$ is a homeomorphism with respect to the subspace topologies.\footnote{In a previous version, we forgot the assumption that $Y$ should be closed. This error was pointed out by Rhea Kommerell, to whom we are very thankful. This assumption is necessary already for~$Y=\Spec(\bbZ)\sminus\{(0)\}$ in~$\Spec(\bbZ)\cong\Spc(\Der(\bbZ)\dd)$, in which case~$\varphi\inv(Y)$ is discrete.}
\smallbreak
\item
\label{it:Spc-Y-b}%
The bijection $\varphi\inv(Y)\to Y$ preserves and reflects specializations:
For $\cQ_1,\cQ_2\in\varphi\inv(Y)$ we have $\cQ_1\in\adhpt{\cQ_2}$ if and only if~$\varphi(\cQ_1)\in\adhpt{\varphi(\cQ_2)}$.
\smallbreak
\item
\label{it:Spc-Y-c}%
The tt-ideal $(\hTd)_{\varphi\inv(Y)}$ of~$\hTd$ supported on~$\varphi\inv(Y)$ is precisely the tt-ideal of compact objects~$(\TYpp)\cc=\TYc$.
\end{enumerate}
\end{Thm}

\begin{proof}
\Cref{Lem:comp-Y} shows that $\varphi$ is a bijection on the subsets~$\varphi\inv(\supp(c))\isoto \supp(c)$ for every~$c\in\TYc$. Since $Y=\cup_{c\in\TYc}\supp(c)$ is the filtered union of these closed subsets, we obtain the bijection in~\eqref{it:Spc-Y-a}. By general tt-geometry, the subcategory of~$\hTd$ supported on $\varphi\inv(Y)$ is precisely the tt-ideal generated by the image of~$\TYc$ under the tt-functor~$\widehat{(-)}\colon \Tc\to \hTd$ we are considering. But we proved in \Cref{Rem:comp-Y-prep} that $\widehat{(-)}$ is an equivalence between $\TYc$ and the compacts in~$\TYpp$, which form a tt-ideal in~$\hTd$. This immediately gives~\eqref{it:Spc-Y-c}. It also follows from \Cref{Lem:comp-Y} that $\varphi\inv(-)$ yields a bijection between the closed subsets of~$Y$ and of~$\varphi\inv(Y)$ of the form~$\supp(c)$ for some~$c\in\TYc$. When $Y$ is closed in~$\SpcT$ these~$\{\supp(c)\}_{c\in\TYc}$ form a closed basis for the subspace topology on~$Y$ and similarly for~$\varphi\inv(Y)$. It follows that the bijection $\varphi\inv(Y)\isoto Y$ given by~$\varphi$ is a homeomorphism in that case. It remains to prove~\eqref{it:Spc-Y-b}. The forward implication ($\then$) is clear by continuity of~$\varphi$. Conversely, suppose that $\cP_1:=\varphi(\cQ_1)$ is a specialization of~$\cP_2:=\varphi(\cQ_2)$ in~$\SpcT$, meaning that $\cP_1\subseteq\cP_2$. Since $\cP_2=\varphi(\cQ_2)\in Y$ there exists $c\in\TYc$ such that $\cP_2$ belongs to the closed subset~$\supp(c)$, and therefore~$\cP_1\in\supp(c)$ as well. It follows from~\eqref{it:Spc-Y-a} and \Cref{Lem:comp-Y} that $\cQ_i=\smash{\fhat{\cP_i}}$ for $i=1,2$. By definition of~$\smash{\fhat{(-)}}$ it follows from~$\cP_1\subseteq \cP_2$ that $\smash{\fhat{\cP_1}}\subseteq\smash{\fhat{\cP_2}}$, that is, $\cQ_1\subseteq \cQ_2$ which reads~$\cQ_1\in\adhpt{\cQ_2}$.
\end{proof}

In view of \Cref{Thm:Spc(compl)-on-Y}, the map $\varphi$ is particularly interesting outside of~$Y$. Recall from \Cref{Def:Tate} that~$\tY(\unit)$ is the Tate ring.
\begin{Cor}
\label{Cor:Spc(compl)}%
	We have $\Img(\varphi)\cap Y^\complement= \Supp(\tY(\unit))$.
\end{Cor}
\begin{proof}
	We have that $\Supp(t)\cap Y^\complement=\Supp(t\otimes\fY)$ for every~$t\in\cT$. This follows from the definition \eqref{eq:Supp(t)} and the fact that $\Supph(\fY)=\pi^{-1}(Y^{\complement})$ by \cite[Lemma~3.8]{BarthelHeardSanders23b}. By \cref{cor:img-of-compl}, we have $\Img(\varphi)=\Supp([\eY,\unit])$. The result then follows since $\tY(\unit)=\fY\otimes[\eY,\unit]$.
\end{proof}

\begin{Rem}
\label{Rem:Supp(tY)}%
	The support of the Tate ring~$\tY(\unit)$ is usually not empty. Indeed, $\Supp(\tY(\unit))=\varnothing$ forces~$\tY(\unit)=0$ by \Cref{rem:detection-for-wring}, which only happens in the `split' case of \Cref{Rem:split}, by \Cref{Prop:Tate-decomposition}. In other words, if~$Y$ is not open and closed, then the support of the Tate ring is non-empty.
\end{Rem}
\begin{Rem}
\label{Rem:patch}%
	The support~$\Supp(\tY(\unit))$ is `proconstructible', that is, it is closed for the constructible topology (\aka the patch topology). A constructible subset of the spectral space~$\SpcT$ is one built from the supports $\supp(c)$ of small objects~${c\in\Td}$ and their complements, by taking finite intersections and finite unions. The proconstructible subsets are the arbitrary intersections of constructible subsets. For instance, the complement of a Thomason subset is proconstructible. More generally, the image $\Img(\varphi)$ of a spectral map~$\varphi$ is proconstructible by~\cite[Corollary~1.3.23]{DickmannSchwartzTressl19}. It follows that $\Supp(\tY(\unit))=Y^\complement\cap\Img(\varphi)$ is always proconstructible.
\end{Rem}

Let us show that the completion of a local tt-category remains local.

\begin{Prop}
\label{Prop:local-compl}%
	Suppose that $\cT$ is local, meaning that $0$ is a prime tt-ideal in~$\Td$, and that our Thomason subset~$Y$ is non-empty. Then $\hTd$ is local. Furthermore the map~$\varphi\colon\Spc(\hTd)\to \SpcT$ preserves the closed point.
\end{Prop}

\begin{proof}
	Let $c,d\in \hTd$ such that $c\,\hat\otimes\,d=0$ in~$\hTd$. This means $[\eY,c\otimes d]=0$ in~$\cT$. Pick $a,b\in\TYc$ compact objects supported on~$Y$. Tensoring $[\eY,c\otimes d]=0$ by~$a\otimes b$ and using that $\eY\otimes a^\vee=a^\vee$, etc., we get
	\begin{equation}
	\label{eq:aux-local-1}%
		(a\otimes c)\otimes (b\otimes d)=0
	\end{equation}
	in~$\cT$. By \Cref{Rem:comp-Y-prep} we know that $a\otimes c$ and~$b\otimes d$ belong to~$\TYc$ and in particular~\eqref{eq:aux-local-1} can be viewed in the local category~$\Tc=\Td$. Now, either $a\otimes c=0$ for every $a\in \TYc$, or not. If not, there is one~$a\in\TYc$ such that $a\otimes c\neq 0$ and then, by locality, $b\otimes d=0$ for all~$b\in\TYc$. By symmetry of the argument, we can assume $a\otimes c=0$ in~$\cT$ for all~$a\in\TYc$. But then, by $\eY\in\Loc{\TYc}$ and cocontinuity of~$[-,c]$ we have $[\eY,c]=0$ which means~$c=0$ in~$\TYpp$. We have shown that $\hTd$ is local.

	Finally, consider any $c \in \Td$ satisfying~$\hat{c}=0$. Since $Y$ is non-empty, there exists a non-zero $b \in \TYc$. We have $b\otimes c \in \TYc$ by \cref{Rem:comp-Y-prep}, so $\widehat{b\otimes c} = 0$ forces~$b\otimes c =0$. Since $\Td$ is local and~$b\neq 0$, this implies~$c=0$. This shows that $\Td\to \hTd$ is conservative, \ie~$\varphi$ maps the closed point~$0$ to the closed point.
\end{proof}

\section{Functoriality}
\label{sec:functorial}

We next consider the functoriality of the constructions of \Cref{sec:recollement}. We keep our \Cref{Hyp}. In particular, we have chosen a Thomason subset~$Y$ in~$\SpcT$.

\begin{Not}
\label{Not:f^*}%
	Let $f^*\colon \cT\to \cS$ be a geometric functor (\cref{Rec:geometric}) and let $f=\Spc(f^*)\colon \SpcS\to \SpcT$. We write $Z\coloneqq f\inv(Y)\subseteq \SpcS$ for the preimage of~$Y$. Recall that $f^*$ admits a coproduct-preserving right adjoint~$f_*\colon \cS\to \cT$, which itself admits a further right adjoint~$f^!\colon \cT\to \cS$ by Brown--Neeman Representability.
\end{Not}
\begin{Prop}
\label{Prop:functorial}%
	Keep the above \Cref{Not:f^*}, in particular $Z=f\inv(Y)$.
\begin{enumerate}[\rm(a)]
\item
\label{it:functorial-f^*}%
	We have canonical isomorphisms $f^*(\eY)\cong \eZ$ and~$f^*(\fY)\cong \fZ$ in~$\cS$. Therefore $f^*(\TY)\subseteq\SZ$ and $f^*(\TYp)\subseteq\SZp$, yielding commutative diagrams
	\[
	\vcenter{\xymatrix@C=3em{
	\TY \ar@{ >->}[r] \ar[d]_-{f^*}
	& \cT \ar@{->>}[r]^-{\eY\otimes-} \ar[d]_-{f^*}
	& \TY \ar[d]^-{f^*}
	\\
	\SZ \ar@{ >->}[r]
	& \cS \ar@{->>}[r]_-{\eZ\otimes-}
	& \SZ
	}}
	\qquadtext{and}
	\vcenter{\xymatrix@C=3em{
	\TYp \ar@{ >->}[r] \ar[d]_-{f^*}
	& \cT \ar@{->>}[r]^-{\fY\otimes-} \ar[d]_-{f^*}
	& \TYp \ar[d]^-{f^*}
	\\
	\SZp \ar@{ >->}[r]
	& \cS \ar@{->>}[r]_-{\fZ\otimes-}
	& \SZp
	}}
	\]
	showing that $f^*$ is compatible with torsion and with local subcategories.
\smallbreak
\item
\label{it:functorial-f_*}%
	We have four canonical isomorphisms of functors $\cS \to \cT$:
	\begin{align*}
	f_*(\eZ\otimes -)\cong \eY\otimes f_*(-)
	& \qquadtext{and}
	f_*(\fZ\otimes -)\cong \fY\otimes f_*(-)
	\\
	f_*([\eZ,-])\cong [\eY,f_*(-)]
	& \qquadtext{and}
	f_*([\fZ,-])\cong [\fY,f_*(-)].
	\end{align*}
	Therefore $f_*(\SZ^\perp)\subseteq\TYp$ and $f_*(\SZpp)\subseteq\TYpp$, yielding commutative diagrams
	\[
	\vcenter{\xymatrix@C=3em{
	\SZ^\perp \ar@{ >->}[r] \ar[d]_-{f_*}
	& \cS \ar@{->>}@<.3em>[r]^-{\fZ\otimes-} \ar@{->>}@<-.3em>[r]_-{[\fZ,-]} \ar[d]_-{f_*}
	& \SZ^\perp \ar[d]^-{f_*}
	\\
	\TYp \ar@{ >->}[r]
	& \cT \ar@{->>}@<.3em>[r]^-{\fY\otimes-} \ar@{->>}@<-.3em>[r]_-{[\fY,-]}
	& \TYp
	}}
	\qquadtext{and}
	\vcenter{\xymatrix@C=3em{
	\SZpp \ar@{ >->}[r] \ar[d]_-{f_*}
	& \cS \ar@{->>}[r]^-{[\eZ,-]} \ar[d]_-{f_*}
	& \SZpp \ar[d]^-{f_*}
	\\
	\TYpp \ar@{ >->}[r]
	& \cT \ar@{->>}[r]_-{[\eY,-]}
	& \TYpp
	}}
	\]
	showing that $f_*$ is compatible with local and with complete subcategories. (The square with~$\ff\otimes-$ and the one with~$[\ff,-]$ commute separately.)
\smallbreak
\item
\label{it:functorial-f^!}%
	We have two canonical isomorphisms of functors~$\cT\to \cS$:
	\[
	[\eZ,f^!(-)]\cong f^!([\eY,-])
	\qquadtext{and}
	[\fZ,f^!(-)]\cong f^!([\fY,-]).
	\]
	Therefore $f^!(\TYpp)\subseteq\SZpp$ and we have a commutative diagram
	\[
	\vcenter{\xymatrix@C=3em{
	\TYpp \ar@{ >->}[r] \ar[d]_-{f^!}
	& \cT \ar@{->>}[r]^-{[\fY,-]} \ar[d]_-{f^!}
	& \TYpp \ar[d]^-{f^!}
	\\
	\SZpp \ar@{ >->}[r]
	& \cS \ar@{->>}[r]_-{[\fZ,-]}
	& \SZpp
	}}
	\]
	showing that $f^!$ is compatible with complete subcategories.
\smallbreak
\item
\label{it:functorial-hat-f^*}%
	Define the (Bousfield) completion of~$f^*$ as the functor~$\hat{f}^*\colon \TYpp\to \SZpp$ given by
	\[
		\hat{f}^*(-):=[\eZ,f^*(-)].
	\]
	Then $\hat{f}^*$ is a tt-functor, left adjoint to (the restriction of) $f_*\colon \SZpp\to \TYpp$ from~\eqref{it:functorial-f_*}. The functor $\hat{f}^*$ restricts to a tt-functor~$\hTd\to \hSd$ on dualizable objects and the following two diagrams commute
	\[
	\vcenter{\xymatrix@C=3em{
	\cT \ar@{->>}[r]^-{[\eY,-]} \ar[d]_-{f^*}
	& \TYpp \ar[d]^-{\hat{f}^*}
	\\
	\cS \ar@{->>}[r]_-{[\eZ,-]}
	& \SZpp
	}}
	\qquadtext{and}
	\vcenter{\xymatrix@C=3em{
	\Td \ar@{->>}[r]^-{\widehat{(-)}{}^{Y}} \ar[d]_-{f^*}
	& \hTd \ar[d]^-{\hat{f}^*}
	\\
	\Sd \ar@{->>}[r]_-{\widehat{(-)}{}^{Z}}
	& \hSd
	}}
	\]
	where $\widehat{(-)}{}^{Y}=[\eY,-]\colon \Td\to \hTd$ is completion along~$Y$ (\Cref{Def:completion}) and $\widehat{(-)}{}^{Z}=[\eZ,-]\colon \Sd\to \hSd$ is completion along~$Z=f\inv(Y)$.
\smallbreak
\item
\label{it:functorial-f_*-Tate}%
	We have a canonical isomorphism
	\begin{equation*}
		\tY\circ {f_*}\cong f_*\circ\ttt_{Z}.
	\end{equation*}
	of lax symmetric monoidal functors $\cS\to \cT$ (see~\Cref{Def:Tate}).
\end{enumerate}
\end{Prop}
\begin{proof}
	In~\eqref{it:functorial-f^*} we have $f^*(\eY)\cong\eZ$ and $f^*(\fY)\cong\fZ$ by~\cite[Theorem~6.3]{BalmerFavi11}. The other statements follow from~\eqref{eq:TY} and~\eqref{eq:TYp}. The first two isomorphisms in~\eqref{it:functorial-f_*} follow from the above and the projection formula. Taking right adjoints in
	\[
		(\eZ\otimes -)\circ f^*\cong f^*\circ (\eY\otimes-)
		\qquadtext{and}
		(\fZ\otimes -)\circ f^*\cong f^*\circ (\fY\otimes-)
	\]
	gives the other two isomorphisms. The other statements follow from~\eqref{eq:TYp} and~\eqref{eq:TYpp}. For~\eqref{it:functorial-f^!}, we can take right adjoints in the first two isomorphisms in~\eqref{it:functorial-f_*}. The other statements follow from~\eqref{eq:TYpp}. Combining two of the isomorphisms in~\eqref{it:functorial-f_*} we obtain
	\begin{align*}
		f_*(\ttt_{Z}(s))& \overset{\textrm{def}}{\ =\ } f_*(\fZ \otimes [\eZ, s]) \overset{\eqref{it:functorial-f_*}}{\ \cong\ }
		\fY\otimes[\eY,f_*(s)]\overset{\textrm{def}}{\ =\ }\tY(f_*(s))
	\end{align*}
	which is the isomorphism of~\eqref{it:functorial-f_*-Tate}. We leave to the reader the tedious verification that this is an isomorphism of lax monoidal functors. In carrying out this verification, it may be helpful to use that the isomorphism $[a,f_*(b)]\simeq f_*[f^*a,b]$ of~\cite[(2.17)]{BalmerDellAmbrogioSanders16} makes the following diagram commute
	\[\xymatrix@=0.95em{
			[a_1,f_* b_1]\otimes[a_2,f_* b_2] \ar[dd] \ar[r]^-{\simeq} & f_*[f^* a_1,b_1]\otimes f_*[f^* a_2,b_2] \ar[r]^-{\mathrm{lax}} & f_*([f^* a_1,b_1]\otimes [f^* a_2,b_2]) \ar[d]\\
																	  &&f_*[f^*a_1 \otimes f^*a_2, b_1\otimes b_2] \ar[d]^{\simeq} \\
		[a_1 \otimes a_2,f_* b_1 \otimes f_* b_2] \ar[r]^-{[1,\mathrm{lax}]} & [a_1\otimes a_2,f_*(b_1\otimes b_2)] \ar[r]^-{\simeq} & f_*[f^*(a_1\otimes a_2),b_1\otimes b_2].
	}\]
	The projection formula satisfies an analogous property. For part~\eqref{it:functorial-hat-f^*}, we need to prove that the composition of tt-functors $[\eZ,-]\circ f^*\colon \cT\to \cS\to \SZpp$ vanishes on the localizing ideal~$\TYp$, for then it factors uniquely as in the left-hand square of~\eqref{it:functorial-hat-f^*}. This holds since~\eqref{it:functorial-f^*} gives the inclusion $f^*(\TYp)\subseteq \SZp=\Ker([\eZ,-])$; see~\eqref{eq:TYpp}. The restriction to dualizable objects is then formal. Finally the adjunction is a direct computation. For every $t\in\TYpp$ and~$s\in\SZpp$, we have $\TYpp(t,f_*(s))=\cT(t,f_*(s))\cong\cS(f^*(t),s)$ since $\TYpp\subseteq\cT$ is a full subcategory and~$f^*\adj f_*$. The last group is also~$\cS(f^*(t),s)\cong\SZpp([\eZ,f^*(t)],s)$ since $[\eZ,-]$ is the left adjoint to the inclusion~$\SZpp\into\cS$ by~\eqref{eq:recollements}. We conclude by the definition of~$\hat{f}^*(t)=[\eZ,f^*(t)]$.
\end{proof}

\begin{Rem}
	We only prove functoriality of the Tate construction for the pushforward~$f_*$ and not for~$f^*$. It is not true that $f^*(\ttt_{Y}(t))$ agrees with $\ttt_{f^{-1}(Y)}(f^*(t))$ even when $f^*$ is a localization. For instance, if $f^*\colon \cT\to \cT/\TY\cong\cat T|_U$ is the localization away from~$Y$ itself, then $f^{-1}(Y)=\emptyset$ and $\ttt_{f^{-1}(Y)}=0$. On the other hand, $f^*(\ttt_{Y}(t))$ `is' essentially $\ttt_Y(t)$ in full, since $\ttt_Y(t)=\fY\otimes[\eY,t]=f_*f^*([\eY,t])$ is local over~$U$.
\end{Rem}

\begin{Rem}
\label{rem:closed-functor}%
	However, if $f^*:\cat T\to \cat S$ is a \emph{closed} functor, meaning in particular $f^*([t,t'])\cong[f^*t,f^*t']$ for every~$t,t'\in\cT$, then the isomorphisms $f^*(\eY)\cong\eYfinv$ and $f^*(\fY)\cong\fYfinv$ give us a natural isomorphism
	\begin{equation*}
		f^*(\tY(-)) \cong \ttt_{f^{-1}(Y)}(f^*(-))
	\end{equation*}
	of lax symmetric monoidal functors $\cat T \to \cat S$. This is the case, for example, if $f^*$ satisfies Grothendieck--Neeman duality; see~\cite[(3.12)]{BalmerDellAmbrogioSanders16}.
\end{Rem}

\begin{Cor}\label{Cor:tate-supp-base-change}
	Using \cref{Not:f^*}, the following hold:
	\begin{enumerate}[\rm(a)]
		\item We have $\Supp(\tY(f_*(\unitS))) = f(\Supp(\ttt_{f^{-1}(Y)}(\unitS)))$ in~$\SpcT$.
\smallbreak
		\item If $f^*$ is a closed functor and $\cat T$ satisfies the steel condition (\Cref{Def:hcomp}) then $\Supp(\ttt_{f^{-1}(Y)}(\unitS)) = f^{-1}(\Supp(\tY(\unitT)))$ in~$\SpcS$.
	\end{enumerate}
\end{Cor}

\begin{proof}
	Part~(a) follows from \cref{Prop:functorial}(\ref{it:functorial-f_*-Tate}) and \cref{Prop:Supp-base-change}\,\eqref{it:Supp-b}. Part~(b) follows from \cref{rem:closed-functor} and \cref{Prop:Supp-base-change}\,\eqref{it:Supp-a}.
\end{proof}

\begin{Rem}
\label{Rem:Tate-functorial}%
	We can precompose the isomorphism of \Cref{Prop:functorial}\,\eqref{it:functorial-f_*-Tate} with~$f^*$ which yields by the projection formula~$f_*\circ f^*\cong f_*(\unit)\otimes -$ an isomorphism
	\[
		\tY(f_*(\unit)\otimes -) \cong f_*(\ttt_{f^{-1}(Y)}f^*(-))
	\]
	of lax symmetric monoidal functors $\cat T\to \cat T$. Let us spell this out for a finite localization.
\end{Rem}

\begin{Cor}
\label{Cor:Tate-loc}%
	Let $W\subseteq \SpcT$ be a Thomason subset and consider $f^*\colon \cT\onto \cS=\cT/\cT_W$ the localization of~$\cT$ away from~$W$. Let $f_*\colon \cS\into\cT$ be the fully faithful right adjoint. (If one identifies $\cS$ with $\cT_W^{\perp}$ as in~\eqref{eq:recollements} then~$f^*=\ff_W\otimes-$ and $f_*=\incl$.) Note that $f\inv(Y)=Y\cap W^\complement$ since $f\colon \SpcS\hook\SpcT$ is the inclusion of~$W^\complement$ (\Cref{Rem:TU}). Then for every $t\in \cT$ we have a natural isomorphism
	\[
		\tY(\ff_W\otimes t) \cong f_*(\ttt_{Y\cap W^\complement}f^*(t))
	\]
	in~$\cT$
	and in particular $f^*(\tY(\ff_W\otimes t))\cong \ttt_{Y\cap W^\complement}f^*(t)$ in~$\cS=\cT\restr{W^\complement}$.
\end{Cor}
\begin{proof}
	We have $f_*(\unit)=\ff_W$ as in~\eqref{eq:recollements} and $f^*\circ f_*\cong \Id_{\cS}$ as with any localization. Applying \Cref{Rem:Tate-functorial} we see that $\tY(\ff_W\otimes t) \cong f_*(\ttt_{Y\cap W^\complement}f^*(t))$. Applying~$f^*$ to both sides gives the second result.
\end{proof}

This allows us to prove an excision result about the Tate construction.
\begin{Thm}
\label{Thm:Tate-excision}%
	Let $W\subseteq \SpcT$ be a Thomason subset such that $Y\cap W=\varnothing$. Let $\cS=\cT/\cT_W=\cT\restr{W^\complement}$ be the $W^\complement$-local tt-category and~$f^*\colon \cT\onto \cS$ the localization. Note that $Y\subseteq W^\complement=\SpcS$, so we can $Y$-complete~$\cS$.
	\begin{enumerate}[\rm(a)]
	\item
	The Bousfield completed tt-functor~$\hat{f}^*\colon \TYpp\onto \SZpp$ of \Cref{Prop:functorial}\,\eqref{it:functorial-hat-f^*} is an equivalence. In particular, it restricts to an equivalence $\hat{f}^*\colon \hTd\isoto \hSd$ on dualizables, between the $Y$-completion of~$\Td$ and the $Y$-completion of~$\Sd$.
	\smallbreak
	\item
	For every $t\in \cT$ the Tate object~$\tY(t)$ is $W^\complement$-local in~$\cT$ and we have a canonical isomorphism
	\[
		\tY(t)\cong f_*(\tY(f^*(t)))
	\]
	where~$f_*$ is the fully faithful right adjoint of~$f^*$. In other words, the Tate functor can be computed $W^\complement$-locally.
	\end{enumerate}
\end{Thm}
\begin{proof}
	Let us temporarily write $Z=f\inv(Y)$ for $Y$ seen inside~$\SpcS$, in accordance with \Cref{Prop:functorial}, where $f\colon \SpcS\hook\SpcT$ is the inclusion onto~$W^\complement$. We claim that the fully faithful $f_*\colon \cS\to \cT$ which we already know restricts to~$\SZpp\to \TYpp$ by \Cref{Prop:functorial}\,\eqref{it:functorial-f_*} defines an inverse to~$\smash{\hat{f}^*}\colon \TYpp\to \SZpp$. Since $f_*$ is fully faithful, it suffices to prove that $f_*\circ\hat{f}^*\simeq \Id_{\TYpp}$.

	By \Cref{Lem:YW=0} the assumption $Y\cap W=\varnothing$ gives $[\eY,\ee_W\otimes t]=0$ for every~$t\in\cT$. This forces $[\eY,t]\cong[\eY,\ff_W\otimes t]$. Since $\ff_W\otimes-\cong f_*f^*\colon \cT\to \cT$, this shows that $[\eY,t]\cong[\eY,f_*f^*t]\cong f_*[\eZ,f^*t]$ by \Cref{Prop:functorial}\,\eqref{it:functorial-f_*}. If we take $t\in\TYpp$ we have $[\eY,t]=t$ and the previous equation reads $t\cong f_*(\hat{f}^*(t))$ by the definition of~$\hat{f}^*$.

	The second part follows similarly from the above $[\eY,t]\cong[\eY,\ff_W\otimes t]$:
	\[
		\tY(t)\overset{\textrm{def}}{\ =\ } \fY\otimes[\eY,t]\cong\fY\otimes[\eY,\ff_W\otimes t]\overset{\textrm{def}}{\ =\ } \tY(\ff_W\otimes t)
	\]
	and the latter is isomorphic to $f_*(\ttt_{Y\cap W^\complement}(f^*(t)))$ by~\Cref{Cor:Tate-loc}. This is the statement since $W^\complement\cap Y=Y$.
\end{proof}

\begin{Cor}
\label{Cor:completion-excision}%
	Let $W\subseteq \SpcT$ be a Thomason subset such that $Y\cap W=\varnothing$. Then the image of $Y$-completion $\Spc(\hTd)\to \SpcT$ is contained in~$W^\complement$.
\qed
\end{Cor}

\begin{Rem}
\label{Rem:W(Y)}%
	For $Y$ fixed, there is a biggest Thomason subset~$W=W(Y)$ such that $Y\cap W=\varnothing$, namely the union of all of them
	\begin{equation}
	\label{eq:W(Y)}%
		W(Y)\coloneqq \bigcup_{W \textrm{ Thomason}\atop W\cap Y=\varnothing}W=\bigcup_{a\in\Td \atop\supp(a)\cap Y=\varnothing}\supp(a).
	\end{equation}
	This Thomason subset is also the support of the kernel of completion:
\end{Rem}
\begin{Prop}
	The support of~$\Ker\big(\widehat{(-)}\colon \Td\to \hTd\big)$ is the~$W(Y)$ of~\eqref{eq:W(Y)}. In other words, for $c\in\Td$ we have $\hat{c}=0$ if and only if~$\supp(c)\cap Y=\varnothing$.
\end{Prop}
\begin{proof}
	If $\hat{c}=0$ then for every~$d\in\TYc$ we have $0=\hat{c}\otimes d\cong c\otimes d$ (see \Cref{Rem:comp-Y-prep}). Thus $\supp(c)\cap \supp(d)=\varnothing$. Since $Y=\cup_{d\in\TYc}\supp(d)$ we get~$\supp(c)\cap Y=\varnothing$. Conversely, if $\supp(c)\cap Y=\varnothing$ then $\hat{c}=[\eY,c]\cong [\eY\otimes c^\vee,\unit]=0$.
\end{proof}

\begin{Exa}
	Suppose that $Y$ contains all the closed points of~$\SpcT$. This holds for example if $\cT$ is local and~$Y\neq\varnothing$. Then the completion tt-functor $\Td\to \hTd$ along~$Y$ is conservative (on dualizable objects).
\end{Exa}

\begin{Rem}
	Completion along~$Y$, like any tt-functor, can be decomposed into a localization followed by a conservative induced functor:
	\[
		\Td \onto \Sd\coloneqq\Td\restr{W(Y)^\complement} \xto{\widehat{(-)}} \hTd\cong\hat{\cS}\dd.
	\]
	Here, \Cref{Thm:Tate-excision} tells us that the second tt-functor is just completion with respect to the `same'~$Y$, using that $Y\subseteq W(Y)^\complement=\SpcS$ is contained in the spectrum of the localization. The $Y$-completion and the Tate ring are~$W(Y)^\complement$-local phenomena. Consequently, we can replace~$\cT$ by~$\cS=\cT\restr{W(Y)^\complement}$ and suppose that $W(Y)=\varnothing$, that is, we can assume that the support of every non-zero dualizable object meets~$Y$. In other words, we can assume that $\widehat{(-)}\colon \Td\to \hTd$ is conservative.
\end{Rem}

\begin{Not}
	For any subset $S \subset \SpcT$, we write
	\[
		\gen(S) \coloneqq \SET{x \in \SpcT}{\overbar{\{x\}} \cap S \neq \emptyset}
	\]
	for the set of generalizations of $S$.
\end{Not}

\begin{Prop}\label{prop:Supp-in-gen}
	The Thomason subset $W(Y)^{\complement}$ of \eqref{eq:W(Y)} coincides with the subset $\gen(\Ycons)$ where~$\Ycons$ denotes the closure of $Y$ in the constructible topology. If the Thomason subset $Y$ is closed or if the space $\SpcT$ is noetherian then $\gen(\Ycons) =\gen(Y)$.
\end{Prop}

\begin{proof}
	Observe that $\cat P \in W(Y)^{\complement}$ if and only if $W \cap Y \neq \emptyset$ for every Thomason subset~$W$ which contains $\cat P$. This is equivalent to saying that $\cat P$ belongs to the closure of $Y$ in the Hochster dual topology $\smash{\Yinv}$. The latter coincides with $\gen(\Ycons)$ by~\cite[Corollary 1.5.5]{DickmannSchwartzTressl19}. For the second part, first note that if $Y$ is closed then it is proconstructible, hence $\Ycons = Y$. Finally, we claim that if $\Spc(\Td)$ is noetherian then $\smash{\Yinv} = \gen(\Ycons)$ is contained in, and hence coincides with,~$\gen(Y)$. Indeed, if $x \not\in \gen(Y)$ then $\overbar{\{x\}} \cap Y = \emptyset$. If the space is noetherian then $\overbar{\{x\}}$ is Thomason hence open in the Hochster dual topology. Thus $\overbar{\{x\}} \cap Y = \emptyset$ implies $x \not\in \smash{\Yinv}$.
\end{proof}

\begin{Rem}
	Thus, if the Thomason subset $Y$ is closed or if $\SpcT$ is noetherian, then every point in $\Supp(\tY(\unit))$ is a generalization of a point in $Y$. The following example shows that in general this need not be the case, which justifies the need for the more general statement of \cref{prop:Supp-in-gen}.
\end{Rem}

\begin{Exa}
	Let $R$ be a non-noetherian absolutely flat ring, such as an infinite product of fields, and consider $\cat T\coloneqq\Der(R)$. For any non-open point $\mathfrak p \in \Spec(R)$, the complement $Y\coloneqq \{\mathfrak p\}^{\complement} = \gen(\mathfrak p)^{\complement}$ is Thomason (but not closed). The image of completion $\Spc(\hTd)\to \Spc(\Td)$ must be everything; otherwise, $\Supp(\tY(\unit))=\emptyset$ so that $Y$ would be open and closed by \cref{Rem:Supp(tY)}, which is false. Thus, in this example $\Supp(\tY(\unit)) = \{\mathfrak p\}$. Note that $\mathfrak p$ is not a generalization of any point in~$Y$. However, in this example $Y$ is constructibly dense, so indeed $\Supp(\tY(\unit)) = \gen(\Ycons) \cap Y^{\complement} = \SpcT \cap Y^{\complement} = Y^{\complement}$.
\end{Exa}

We can also illustrate this phenomenon in another area of mathematics:
\begin{Exa}
	Let $\bbF$ be a finite field and let~$\cT$ be the derived category of Artin motives over~$\bbF$ with coefficients in a field~$k$ of positive characteristic. This big tt-category is also the derived category of permutation modules over the absolute Galois group of~$\bbF$; see~\cite{BalmerGallauer25}. Its spectrum of dualizables looks as follows
	\[
	\xymatrix@C=.5em@R=1em{
		\cM_0 && \cM_1 && \cdots && \cM_n && \cM_{n+1} && \cdots && \cM_{\infty}
		\\
		& \cP_1 \ar@{~>}[lu] \ar@{~>}[ru]
		&& \ar@{~>}[lu]
		& \cdots
		& \ar@{~>}[ru]
		&& \cP_n \ar@{~>}[lu] \ar@{~>}[ru]
		&& \cP_{n+1} \ar@{~>}[lu] \ar@{~>}[ru]
		& \cdots
	}
	\]
	See details in \cite[Theorem~1.4]{BalmerGallauer25}. In particular, the point~$\cM_{\infty}$ is closed but not Thomason. Its complement $Y=\SET{\cM_n}{0\le n<\infty}\cup\SET{\cP_n}{1\le n<\infty}$ is a Thomason subset which is not closed, for $\cM_{\infty}$ belongs to its closure. It follows that~$\cM_{\infty}$ belongs to the support of the Tate ring~$\tY(\unit)$ as $Y^\complement=\{\cM_\infty\}$ and the category is not split. Note that $\cM_{\infty}$ is not a generalization of any point in~$Y$. This is an example where $\SpcT$ is not noetherian and $Y$ is not closed. On the other hand, $\cM_{\infty}$ is contained in~$\Ycons$.
\end{Exa}

\begin{Rem}
	We have established that in general
	\[
		\Img(\varphi_Y) \subseteq W(Y)^{\complement}=\gen(\Ycons).
	\]
	Several examples given below (e.g., \cref{exa:valuation-domain}) show that this can be a strict inclusion.
\end{Rem}

\section{The Tate Intermediate Value Theorem}
\label{sec:TIV}

We keep \Cref{Hyp}: We have a `big' tt-category~$\cT$ and a Thomason subset~$Y\subseteq\SpcT$ of its spectrum. We write $\widehat{(-)}\colon \Td\to \hTd$ for the completion along~$Y$ as in \Cref{Def:completion} and~$\varphi=\varphi_Y\colon\Spc(\hTd)\to\SpcT$ for the induced map on spectra. Recall from \Cref{Thm:Spc(compl)-on-Y} that $\varphi$ is a bijection above~$Y$.

\begin{Thm}[Tate Intermediate Value Theorem, Strong Form]
\label{Thm:TIV}%
	Let $\cP_1\in Y$ and let $\cQ_1\in\Spc(\hTd)$ be its unique preimage in the spectrum of completion. Let $\cP_0\in \SpcT$ be a generalization of~$\cP_1$ which is not contained in~$Y$, that is, $\cP_1\in\adhpt{\cP_0}$ and~$\cP_0\notin Y$. Then there exists a point $\cQ\in\Spc(\hTd)$, which is a generalization of~$\cQ_1$ and which is not contained in~$\varphi\inv(Y)$, whose image under~$\varphi$ is an intermediate specialization between~$\cP_0$ and~$\cP_1$, that is, $\varphi(\cQ)\in\adhpt{\cP_0}$ and $\cP_1\in\adhpt{\varphi(\cQ)}$.

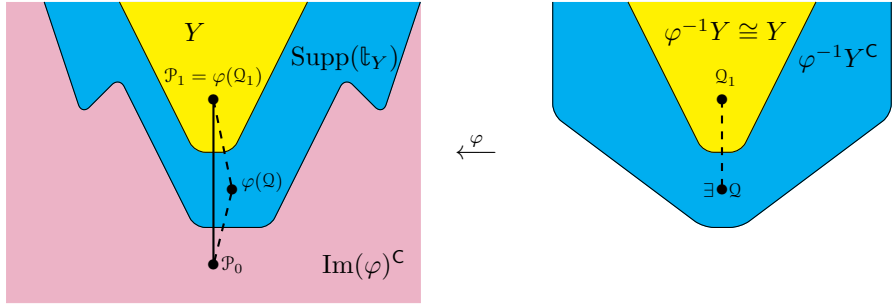
	\[
	\vcenter{\hbox{
		\begin{tikzpicture}
		\filldraw[color=purple!30] (-1.5,0) -- (-1.5,-4) --(4.0,-4) -- (4,0);
		\filldraw[color=cyan,rounded corners] (-1,0) -- (-.5,-1.5) -- (0,-1.0)--(1.0,-3) --(2,-3) -- (3,-1.0)-- (3.5,-1.5)-- (4,0);
		\draw[color=black,rounded corners] (-1,0) -- (-.5,-1.5) -- (0,-1.0)--(1.0,-3) --(2,-3) -- (3,-1.0)-- (3.5,-1.5)-- (4,0);
	\filldraw[color=yellow,rounded corners] (0,0) -- (1,-2) --(1.5,-2) -- (2.5,0);
	\draw[color=black,rounded corners] (0,0) -- (1,-2) --(1.5,-2) -- (2.5,0);
	\node at (1,-0.4) {$Y$};
	\draw (1.25,-1.3) node {$\bullet$};
	\draw[thick] (1.25,-1.3) -- (1.25,-3.5);
	\node[scale=.75] at (1.25,-1) {$\cat P_1=\varphi(\cQ_1)$};
	\draw (1.5,-2.5) node {$\bullet$}; \node[scale=.75] at (1.9,-2.4) {$\varphi(\cat Q)$}; \draw[thick,dashed] (1.25,-1.3) -- (1.5,-2.5) -- (1.25,-3.5);
	\node at (3,-.75) {$\Supp(\tY)$};
	\node at (3.25,-3.5) {$\Img(\varphi)^\complement$};
	\draw (1.25,-3.5) node {$\bullet$}; \node[scale=.75] at (1.5,-3.5) {$\cat P_0$};
	\end{tikzpicture}}}
	\quad\overset{\varphi}{\longleftarrow}
		\vcenter{
			\hbox{
		\begin{tikzpicture}
		\filldraw[color=white!0] (-1.5,0) -- (-1.5,-4) --(4.0,-4) -- (4,0);
		\filldraw[color=cyan,rounded corners] (-1,0) -- (-1,-1.5) -- (1.0,-3) --(1.5,-3) --  (3.5,-1.5)-- (3.5,0);
		\draw[color=black,rounded corners] (-1,0) -- (-1,-1.5) -- (1.0,-3) --(1.5,-3) --  (3.5,-1.5)-- (3.5,0);
	\filldraw[color=yellow,rounded corners] (0,0) -- (1,-2) --(1.5,-2) -- (2.5,0);
	\draw[color=black,rounded corners] (0,0) -- (1,-2) --(1.5,-2) -- (2.5,0);
	\draw (1.25,-1.3) node {$\bullet$}; \node[scale=.75] at (1.3,-1) {$\cat Q_1$};
	\draw (1.25,-2.5) node {$\bullet$};\node[scale=.75] at (1.425,-2.5) {$\cat Q$};
	\node at (1.3,-0.4) {$\varphi\inv Y\cong Y$};
	\node at (2.8,-.75) {$\varphi\inv Y^\complement$};
	\node[scale=.75] at (1.075,-2.5) {$\exists$};
	\draw[thick,dashed] (1.25,-1.3) -- (1.25,-2.5);
		\end{tikzpicture}}}
	\]
\end{Thm}

\begin{proof}
	We begin by reducing to the case where $\Td$ (and therefore~$\hTd$ by \Cref{Prop:local-compl}) is local, and where $\cP_1=0$ and~$\cQ_1=0$ are the closed points.

	To this end, let us use localization away from a Thomason subset~$W\subset\SpcT$. We shall apply this to $W=\Supp(\cP_1)=\gen(\cP_1)^{\complement}$ but this choice is not essential right now. We simply assume that $\cP_0$ and~$\cP_1$ belong to~$W^\complement$, which is equivalent (since $\cP_1\subseteq\cP_0$) to $\Td_W\subseteq\cP_1$.

	Let us use functoriality of completion, in the form of~\Cref{Prop:functorial}\,\eqref{it:functorial-hat-f^*}, for the localization functor $f^*\colon \cT\onto \cS\coloneqq\cT\restr{W^{\complement}}=\cT/\cT_W$. Recall that in this case~$f=\Spc(f^*)\colon \SpcS\hook\SpcT$ is the inclusion of~$W^\complement$ inside~$\SpcT$ and therefore $Z\coloneqq f\inv(Y)$ is simply $Y\cap W^\complement$. \Cref{Prop:functorial}\,\eqref{it:functorial-hat-f^*} provides the left-hand commutative square of tt-categories below
	\begin{equation}
	\label{eq:Q-compl-Spc}%
	\vcenter{
		\xymatrix@C=4em{
	\Td \ar[r]^-{\widehat{(-)}{}^Y}  \ar[d]_-{f^*}
	& \hTd \ar[d]^-{\hat{f}^*} &
	\Spc(\Td) \ar@{<-}[r]^-{\varphi_Y} \ar@{<-^)}[d]_-{f}
	& \Spc(\hTd) \ar@{<-}[d]^-{\Spc(\hat{f}^*)}
	\\
	\Sd \ar[r]^-{\widehat{(-)}{}^Z}
	& \hSd &
	\Spc(\cS^{\dname^{\vphantom{I}}}) \ar@{<-}[r]^-{\varphi_Z}
	& \Spc(\hSd)
	}}
	\end{equation}
	whose image under~$\Spc(-)$ is the right-hand commutative square.

	Now observe the following:
\begin{enumerate}[(1)]
\item
\label{it:redux-to-local-1}%
	When viewed in~$\SpcS=W^\complement$, the point~$\cP_1$ belongs to~$Z=Y\cap W^\complement$ and the point~$\cP_0$ lies outside of~$Z$.
\item
\label{it:redux-to-local-2}%
	By \cref{Thm:Spc(compl)-on-Y}, there exists a unique point~$\cR_1\in\Spc(\hSd)$ with ${\varphi_Z(\cR_1)=\cP_1}$. By the commutativity of~\eqref{eq:Q-compl-Spc} and the uniqueness of the preimage of~$\cP_1\in Y$ under~$\varphi_Y$, the image of~$\cR_1$ under~$\Spc(\hat{f}^*)$ must be our~$\cQ_1$.
\item
\label{it:redux-to-local-3}%
	If there exists a point $\cR$ in~$\Spc(\hSd)$ such that $\cR_1\in\adhpt{\cR}$ and~$\varphi_Z(\cR)\in\adhpt{\cP_0}$ in~$\SpcS$, then there exists a point~$\cQ$ in~$\Spc(\hTd)$ such that $\cQ_1\in\adhpt{\cQ}$ and~$\varphi_Y(\cQ)\in\adhpt{\cP_0}$ in~$\SpcT$. Indeed, in view of the commutativity of~\eqref{eq:Q-compl-Spc}, it suffices to define~$\cQ$ as the image of~$\cR$ under~$\Spc(\hat{f}^*)$.
\item
\label{it:redux-to-local-4}%
	Let $\cR\in\Spc(\hSd)$ and let $\cQ\in\Spc(\hTd)$ be its image under~$\Spc(\hat{f}^*)$. If $\varphi_Z(\cR)$ is not in~$Z$ then $\varphi_Y(\cQ)$ is not in~$Y$. Indeed, if $Y\ni \varphi_Y(\cQ)=f(\varphi_Z(\cR))$ then $\varphi_Z(\cR)\in f\inv(Y)=Y\cap W^\complement=Z$.
\end{enumerate}

	In other words, it is enough to prove the theorem for this~$\cS=\cT\restr{W^\complement}$, as long as~$W^\complement$ contains~$\cP_1$ (and therefore~$\cP_0$). The best such~$W$ is of course~$W=\Supp(\cP_1)$, in which case $W^{\complement} = \gen(\cat P_1)$ is local and $\cP_1/\Td_{W}=0$. So we can indeed assume that~$\Td$ is local and that $\cP_1=0$, in which case its unique preimage under~$\varphi_Y$ is~$\cQ_1=0$ by \cref{Prop:local-compl}.

\smallbreak
	\emph{We assume that~$\cP_1=0$ and~$\cQ_1=0$ and that~$\Td$ and~$\hTd$ are local.}
\smallbreak

	In particular, all primes~$\cQ$ that we construct in~$\Spc(\hTd)$ are generalizations of~$\cQ_1$ so we do not need to worry about this property anymore. We only need to find a prime~$\cQ\in\Spc(\hTd)$ whose image under~$\varphi_Y$ lies outside of~$Y$ and is a specialization of~$\cP_0$.

	We consider two classes of objects in~$\hTd$. First, the $\otimes$-multiplicative collection
	\[
		S \coloneqq \SET{\hat{b}}{b \in \Td, b\not\in \cP_0}
	\]
	which is the image of the complement of~$\cP_0$ and, secondly, the tt-ideal
	\[
		\cat J \coloneqq \SET{\hat{c}}{c\in\TYc}
	\]
	corresponding to~$\TYc$ in~$\hTd$, by \Cref{Lem:A-hat}. The key fact is that they are disjoint:
\begin{Claim}
\label{claim:S-J}%
	With the above notation, we have $S\cap \cJ=\varnothing$.
\end{Claim}
\noindent
	Indeed, suppose, \textsl{ab absurdo}, that there exists $b\in\Td\sminus\cP_0$ and $c\in\TYc$ such that~$\hat{b}\simeq\hat{c}$ in~$\hTd$. Applying the left adjoint~$\eY\otimes-$ to completion, as in~\eqref{eq:recollements}, and using that $\eY\otimes[\eY,-]\cong\eY\otimes-$ and that $\eY\otimes c\cong c$ since $c\in\TYc$ we deduce that
	\[
		\eY\otimes b\simeq c
	\]
	in~$\cT$. In particular, $\eY\otimes b$ is dualizable and since $b$ is as well, the exact triangle $\eY\otimes b\to b\to \fY\otimes b\to \Sigma\eY\otimes b$ tells us that so is~$\fY\otimes b$. Now, the dualizable objects~$\eY\otimes b$ and~$\fY\otimes b$ tensor to zero so one of them must be zero since~$\Td$ is local. Both cases lead to a contradiction, using the assumption $b\notin\cP_0$ which reads $\cP_0\in\supp(b)$. First, if $\fY\otimes b=0$ then $\supp(b)=\supp(\eY\otimes b)\subseteq Y$ gives us $\cP_0\in\supp(b)\subseteq Y$ which contradicts our assumption on~$\cP_0$. Second, if $\eY\otimes b=0$ then $\supp(b)=\supp(\fY\otimes b)\subseteq Y^\complement$ and the relation $\cP_1\in\adhpt{\cP_0}\subseteq\supp(b)\subseteq Y^\complement$ contradicts our assumption on~$\cP_1$. This proves \Cref{claim:S-J}.

	By \Cref{Rec:brothers}\,\eqref{it:brother-1} for the tt-category~$\hTd$, the disjunction between the \mbox{$\otimes$-}multiplicative set~$S$ and the tt-ideal~$\cJ$ must be witnessed by a prime~$\cQ\in\Spc(\hTd)$, satisfying $S\cap \cQ=\varnothing$ and $\cJ\subseteq\cQ$. By the definition of~$S$ and~$\cJ$, this means that
	\begin{equation*}
	\label{eq:aux-8}%
		b\in\Td\sminus\cP_0\ \then \ \hat{b}\notin\cQ
		\qquadtext{and}
		c\in\TYc\ \then \ \hat{c}\in\cQ.
	\end{equation*}
	In other words, the image of~$\cQ$ under~$\varphi_Y=\Spc(\widehat{(-)})$, namely $\SET{d\in\Td}{\hat{d}\in\cQ}$, is contained in~$\cP_0$ and contains~$\TYc$. The former means that $\varphi_Y(\cQ)$ is a specialization of~$\cP_0$ and the latter means that $\varphi_Y(\cQ)$ does not belong to~$Y$.
\end{proof}

We now record the weak form of the Tate Intermediate Value Theorem:
\begin{Cor}
\label{Cor:TIV}%
	Let $Y\subseteq \SpcT$ be a Thomason subset and let $\cP_0,\cP_1\in \SpcT$ be two primes such that $\cP_1$ belongs to~$Y$ while $\cP_0$ lies outside of~$Y$. If $\cP_1$ belongs to the closure of~$\cP_0$ in~$\SpcT$ then there exists an `intermediate' point~$\cP_{0.5}$
	\[
		\cP_1\in\adhpt{\cP_{0.5}}\qquadtext{and} \cP_{0.5}\in \adhpt{\cP_0}
	\]
	that furthermore belongs to the support of the Tate ring: $\cP_{0.5}\in\Supp(\tY(\unit))$.
\end{Cor}
\begin{proof}
	This is direct from \cref{Thm:TIV} and \cref{Cor:Spc(compl)}.
\end{proof}

\begin{Rem}\label{rem:immediate-generalization}
	We say that a point $x \in Y^{\complement}$ is an \emph{immediate generalization} of $Y$ if there exists $y \in Y$ such that $\gen(\{y\}) \cap \overbar{\{x\}} = \{x,y\}$.
\end{Rem}

\begin{Cor}\label{Cor:immediate-generalization}
	Every $x\in Y^\complement$ that is an immediate generalization of~$Y$ belongs to~$\Supp(\tY(\unit))$. In other words, $\Img(\Spc(\hTd)\to\SpcT)$ contains not only the subset~$Y$ but also all its immediate generalizations.
\end{Cor}
\begin{proof}
	Let $y\in Y$ such that $x\leadsto y$ and~$\gen(\{y\})\cap\adhpt{x}=\{x,y\}$. This means that $x\leadsto z\leadsto y$ forces $z=x$ or~$z=y$. The result now follows from \Cref{Cor:TIV}.
\end{proof}

We just proved that $\Img(\varphi)$ is bounded below by $Y$ and its immediate generalizations. On the other hand, here is a statement about upper bounds for $\Img(\varphi)$.
\begin{Prop}\label{prop:uniformly-nilpotent}
	Let $\cat T$ be a rigidly-compactly generated tt-category and let
	\[\varphi \colon \Spc(\hTd) \to \Spc(\Td)\]
	be the map induced by completion with respect to the Thomason subset $Y\subseteq \Spc(\Td)$. Let $c \in \Td$ and consider the exact triangle
	\begin{equation}\label{eq:nil-fiber-seq}
		w_c \xrightarrow{\xi_c} \unit \xrightarrow{\mathrm{coev}} c \otimes c^{\vee} \to \Sigma w_c.
	\end{equation}
	The following are equivalent:
	\begin{enumerate}[\rm(i)]
		\item\label{it:unif-1} The image of $\varphi$ is contained in $\supp(c)$.
		\item\label{it:unif-2} The map $\xi_c$ is nilpotent on $\eY$: there exists $n \ge 1$ such that $\xi_c^{\otimes n} \otimes \eY=0$.
	\end{enumerate}
\end{Prop}

\begin{proof}
	First recall from~\cite[Proposition 2.10]{Balmer18} that the collection $\Nil_{\cT}(\xi_c)$ of objects of~$\cT$ on which $\xi_c$ is nilpotent forms a thick ideal of the big category~$\cat T$:
	\[
		\Nil_{\cT}(\xi_c) = \langle c \rangle_{\cT} = \langle \cone(\xi_c) \rangle_{\cT} = \langle \cone(\xi_c^{\otimes n})\rangle_{\cT}
	\]
	for all $n \ge 1$. Next recall that $\varphi^{-1}(\supp(d)) = \supp(\hat{d})$ for any $d \in \Td$. This immediately implies that $\Img(\varphi) \subseteq \supp(d)$ if and only if $\hat{d}\in \TYc$ is fully supported, $\supp(\hat{d}) = \Spc(\hTd)$, which is the case if and only if $\hat{\unit}$ belongs to the thick ideal of $\hTd$ generated by $\hat{d}$. If \eqref{it:unif-2} holds then $\eY$ is a direct summand of $\eY \otimes \cone(\xi_c^{\otimes n})$ for some $n \ge 1$. Setting $d\coloneqq \cone(\xi_c^{\otimes n})$, it follows that $\hat{\unit}$ is a direct summand of $\hat{d}$ in~$\hTd$. Hence $\supp(\hat{d}) = \Spc(\hTd)$ so that $\Img \varphi \subseteq \supp(d)$. As noted above $\langle d\rangle = \langle c \rangle$ which gives~\eqref{it:unif-1}. Conversely, if \eqref{it:unif-1} holds then $\hat{\unit}$ is in the thick ideal generated by $\hat{c}$. This implies that in $\cat T$ we have $\eY \in \Thick\langle \eY \otimes c \otimes \TYc \rangle$. Since $\Nil_\cT(\xi_c)$ is a thick ideal and contains $c$, it must therefore contain $\eY$, which gives~\eqref{it:unif-2}.
\end{proof}

\begin{Cor}\label{cor:nil-ring}
	Let $R$ be a commutative ring which contains two elements $a,b \in R$ such that $a^n$ divides $b$ for all $n \ge 1$. Let $\cat T=\Der(R)$ and $Y=V(a)=\supp(\cone(a))$. Then the image of $\varphi\colon \Spc(\hTd)\to\Spc(\Td)$ is contained in $V(b)$.
\end{Cor}

\begin{proof}
	The idempotent~$\ee_{V(a)}\in\Der(R)$ is well-known to be the homotopy colimit of the following sequence of maps $\cone(a^n)[-1]\to \cone(a^{n+1})[-1]$, for $n\ge0$:
	\[
	\xymatrix@R=1em{
		\cone(a^n)[-1] \ar[d]
		&& \cdots \ar[r]
		& 0 \ar[r]
		& R \ar[r]^-{a^n} \ar@{=}[d]
		& R \ar[r] \ar[d]^-{a}
		& 0 \ar[r]
		& \cdots
		\\
		\cone(a^{n+1})[-1]
		&& \cdots \ar[r]
		& 0 \ar[r]
		& R \ar[r]^-{a^{n+1}}
		& R \ar[r]
		& 0 \ar[r]
		& \cdots
	}
	\]
	In other words, $\eY$ for $Y=V(a)$ fits in an exact triangle
	\[
		\coprod_{n\in\bbN}\cone(a^n)[-1] \to
		\coprod_{n\in\bbN}\cone(a^n)[-1] \to
		\eY\to
		\coprod_{n\in\bbN}\cone(a^n).
	\]
	Since $b$ divides every~$a^n$, it is easy to see that $b\otimes\id_{\cone(a^n)}$ is zero on each summand~$\cone(a^n)[\ast]$ and therefore on~$\coprod_{n\in\bbN}\cone(a^n)[\ast]$ as well. It follows that $b\otimes\eY$ is zero (thanks to the model, or we can conclude that it squares to zero just using the above triangle). Setting $c \coloneqq \cone(b\colon \unit\to \unit)$, in the notation of~\eqref{eq:nil-fiber-seq} we have $\xi_c \otimes c=0$, which implies that $\xi_c:w_c\to \unit$ factors through $b:\unit \to \unit$. Hence, $\xi_c\otimes\eY$ is zero and we conclude by \cref{prop:uniformly-nilpotent} \eqref{it:unif-2}$\then$\eqref{it:unif-1}.
\end{proof}

\begin{Exa}\label{exa:valuation-domain}
	Let $R$ be a valuation domain with valuation~$\nu\colon R \to \mathbb{Z}^2\cup\{\infty\}$ where $\bbZ^2$ has the lexicographic order; see~\cite[Chapter II]{FuchsSalce01}, for example. The spectrum of~$R$ consists of three points: the zero prime ideal $0$, a middle prime ideal $\mathfrak p=\SET{{r\in R}}{\nu(r)\in(\bbZ_{>0}\times\bbZ)\cup\{\infty\}}$ and a principal maximal ideal~$\mathfrak m=\SET{{r\in R}}{{\nu(r)>(0,0)}}$. Let~$a$ be a generator of the maximal ideal, which has valuation $(0,1)$, and let $b$ be any non-zero element of the middle prime, for example, an element which has valuation~$(1,0)$. Then the map on completion with respect to~$Y=V(a)=\{\gm\}$ is the inclusion of the two top points:
	\begin{equation}\label{eq:val-domain}
\begin{tikzpicture}[baseline={([yshift=-1.4ex]current  bounding  box.center)}]
	\coordinate (base) at (0,0);
	\coordinate (start) at ($(base)+(-2*0.5,0)$);
	\node at ($(start)+(0,0.5)$) {$\Spc(\hTd)$};
	\draw (start) -- ++($(0,-0.55)$);
	\filldraw ($(start)$) circle (0.06cm);	
	\filldraw ($(start)+(0,-0.55)$) circle (0.06cm);

		\draw[right hook->] ($(base)+(-0.5,-0.5*0.55)$) --
		($(base)+(0.5,-0.5*0.55)$);

	\coordinate (start) at ($(base)+(2*0.5,0)$);
	\node at ($(start)+(0,0.5)$) {$\Spc(\Td)$};
	\node at ($(start)+(1.25,0)$) {$Y$};
	\node at ($(start)+(1.25,-0.55)$) {$\Supp(\tY)$};
	\draw[fill=yellow!60] (start) circle (3*0.06cm);
	\draw[fill=cyan] ($(start)+(0,-0.55)$) circle (3*0.06cm);
	\draw (start) -- ++($(0,-2*0.55)$);
	\filldraw ($(start)$) circle (0.06cm);	
	\filldraw ($(start)+(0,-0.55)$) circle (0.06cm);
	\filldraw ($(start)+(0,-2*0.55)$) circle (0.06cm);
\end{tikzpicture}
	\end{equation}

	\smallskip
	\noindent Indeed, $a^n \mid b$ for all $n \ge 1$ since $\nu(a^n)=(0,n) \le (1,0)=\nu(b)$. Hence, \cref{cor:nil-ring} implies that the image of~$\varphi_{V(a)}$ is contained in $V(b)$ which does not contain the generic point since $b$ is not nilpotent. Moreover, $\Supp(\tY(\unit))$ must contain the middle point by \cref{Cor:immediate-generalization}. On the other hand, the $\mathfrak m$-adic completion~$\hat{R}_{\mathfrak m}$ is a discrete valuation ring by~\cite[Theorem 7]{KangPark99} and~\cite[\href{https://stacks.math.columbia.edu/tag/05GH}{Lemma 05GH}]{stacks-project}. Finally, since the maximal ideal $\mathfrak m=(a)$ is Koszul-complete, \eqref{eq:val-domain} can be identified with the map $\Spec(\hat{R}_{\mathfrak m}) \to \Spec(R)$ induced by classical $\mathfrak m$-adic completion (\cref{Exa:hatR_I}). It follows that the map~$\varphi_Y$ is indeed as depicted above.
\end{Exa}

\section{The gluing of completion and localization}
\label{sec:topology}

The Tate Intermediate Value Theorem of the previous section shows how the topological interaction between a Thomason subset and its complement is mediated through the support of the associated Tate ring. From a slightly different point of view, the theorem enables us to completely determine the inclusions among the primes of $\Td$ provided we understand such inclusions for the completion $\smash{\hTd}$ along~$Y$ and for the localization $\smash{\TU\dd}$ onto its complement~$U=Y^\complement$. In particular, keeping the notation of \cref{sec:recollement}, we have:

\begin{Cor}\label{cor:recover-specializations}
	A subset $V \subseteq \Spc(\Td)$ is specialization-closed if and only if the following two conditions are satisfied:
	\begin{enumerate}[\rm(i)]
		\item $V \cap U$ is a specialization-closed subset of $U \cong \Spc(\TUd)$ and
		\item $\varphi^{-1}(V)$ is a specialization-closed subset of $\Spc(\hTd)$.
	\end{enumerate}
\end{Cor}

\begin{proof}
	The ($\Rightarrow$) direction is immediate since the preimage of any specialization-closed subset along any continuous map remains specialization-closed. To establish the ($\Leftarrow$) direction, suppose $x \in V$ and $x \leadsto y$. We need to prove $y \in V$. If $y \not\in Y$ (hence also $x \not\in Y$) then this follows from (i). On the other hand, if $x \in Y$ (hence also $y \in Y$), then since $\varphi^{-1}(Y) \to Y$ is a continuous bijection that reflects specialization (\cref{Thm:Spc(compl)-on-Y}), we have $x' \leadsto y'$ for the unique points $x',y' \in \smash{\Spc(\hTd)}$ with $\varphi(x')=x$ and $\varphi(y')=y$. Since $x' \in \varphi^{-1}(V)$, (ii) implies $y' \in \varphi^{-1}(V)$ so that $y \in V$. It remains to consider the case where $y \in Y$ and $x \not\in Y$, which is the setup of the Tate Intermediate Value Theorem. So $y=\varphi(y')$ for a unique $y'\in \Spc(\hTd)$. By \cref{Thm:TIV}, there exists $z' \in \Spc(\hTd)$ such that $z' \leadsto y'$, $\varphi(z') \not\in Y$ and $x \leadsto \varphi(z') \leadsto y$. By (i), we have $\varphi(z') \in V$. Hence $z' \in \varphi^{-1}(V)$. Thus (ii) implies $y' \in \varphi^{-1}(V)$ so that $y \in V$, as desired.
\end{proof}

\begin{Rem}
	If the spectrum $\Spc(\Td)$ is noetherian then its topology is completely determined by the inclusions among primes, but in general this need not be the case. For example, a non-noetherian absolutely flat ring, such as an infinite product of fields, has trivial specialization order (it has dimension zero) but its spectrum is not discrete.
\end{Rem}

\begin{Rem}
\label{Rem:Eilvis}%
	In general, the topology of a spectral space $X$ such as $X=\Spc(\Td)$ is completely determined by the specializations among its points \emph{together with} the associated constructible topology on $X$. This is formalized by the isomorphism between the category of spectral spaces and the category of so-called Priestley spaces; see~\cite[Theorem~1.5.4]{DickmannSchwartzTressl19}. The closed subsets of $X$ are precisely the subsets which are closed in the constructible topology (\aka proconstructible) and specialization-closed.
\end{Rem}

\begin{Cor}\label{cor:closed-determined}
	A subset $Z \subseteq \Spc(\Td)$ is closed if and only if the following two conditions are satisfied:
	\begin{enumerate}[\rm(i)]
		\item $Z \cap U$ is a closed subset of $U\cong \Spc(\TUd)$ and
		\item $\varphi^{-1}(Z)$ is a closed subset of $\Spc(\hTd)$.
	\end{enumerate}
\end{Cor}

\begin{proof}
	The $(\Rightarrow)$ direction is immediate. For the ($\Leftarrow$) direction, \Cref{cor:recover-specializations} implies that $Z$ is specialization-closed. By \Cref{Rem:Eilvis}, it remains to prove that $Z$ is proconstructible. Note that the complement $U=Y^\complement$ is a proconstructible subset of $X\coloneqq \Spc(\Td)$ (for example, because it is closed in the Hochster dual topology). Hence by~\cite[Theorem 2.1.3]{DickmannSchwartzTressl19}, (i) implies that $Z \cap U = W \cap U$ for some proconstructible subset $W \subseteq X$. It follows that $Z\cap U$ is a proconstructible subset of $X$ being the intersection of two proconstructible sets. On the other hand, (ii) implies that $Z \cap \Img(\varphi) = \varphi(\varphi^{-1}(Z))$ is proconstructible since it is the image of a proconstructible set under a spectral map; see~\cite[Corollary 1.3.23]{DickmannSchwartzTressl19}. Hence $Z = Z \cap (\Img(\varphi) \cup U) = (Z \cap \Img(\varphi)) \cup (Z \cap U)$ is proconstructible, as desired.
\end{proof}

\begin{Cor}
\label{Cor:pushout}%
	Let $V\coloneqq\varphi\inv(U)\subset\Spc(\hTd)$ be the preimage of the complement of~$Y$ in the completion along~$Y$. The commutative square
	\[
	\xymatrix{
		\SpcT & \Spc(\hTd) \ar[l]_-{\varphi}
		\\
		U \vbuff \ar@{^(->}[u]
		& V=\varphi\inv(U) \vbuff \ar@{^(->}[u] \ar[l]^-{\varphi\restr{V}}
	}
	\]
	is a pushout of topological spaces.
\end{Cor}
\begin{proof}
	It follows from \cref{Thm:Spc(compl)-on-Y} that it is a pushout of sets and \cref{cor:closed-determined} then establishes that it is a pushout of topological spaces.
\end{proof}
\begin{Rem}
	The commutative square of \Cref{Cor:pushout} is also a pushout in the category of spectral spaces and spectral maps (\ie continuous maps for which the preimage of a quasi-compact open remains quasi-compact). This is a general fact about a \emph{finite} diagram of spectral spaces: If a spectral space $X$ is the colimit in topological spaces of a finite diagram~$X_i$ of spectral spaces and spectral maps, then~$X$ is also the colimit of the~$X_i$ in the category of spectral spaces. This is a straightforward exercise based on the fact that a subset of~$X$ is quasi-compact if its preimage in every~$X_i$ is quasi-compact.
\end{Rem}
\begin{Rem}
	The commutative square of \Cref{Cor:pushout} is the image under $\Spc(-)$ of the commutative square of tt-categories
	\[
	\xymatrix@C=4em{
		\Td \ar[d]_-{(-)\restr{U}} \ar[r]^-{\widehat{(-)}}
		& \hTd \ar[d]^-{(-)\restr{V}}
		\\
		\Td\restr{U} \ar[r]^-{\big(\widehat{(-)}\big)\restr{U}}
		& {{\hTd}}\restr{V}
	}
	\]
	where the bottom arrow is defined as the localization of the top one over~$U$ (not as a completion). This compares to the pullback in~\cite[Theorem~5.11]{NaumannPolRamzi24}.
\end{Rem}

\section{Examples}
\label{sec:examples}%

Let us review a broad range of examples, grouped into four subsections.

\subsection{Chromatic examples}

Recall that the category of spectra $\SH=\Ho(\Sp)$ has spectrum consisting of points $\cat C_{p,n} \coloneqq \SET{x \in \SHd}{K(p,n-1)_*(x)=0}$ ranging over the primes $p$ and `chromatic' integers $1 \le n \le \infty$.

\begin{Rem}\label{rem:supp-in-SH}
	The category of spectra $\SH$ satisfies the steel condition (\cref{Def:hcomp}) and the homological residue fields are provided by the Morava $K$-theories; see~\cite{BalmerCameron21}. It follows that for any weak ring $A$ in $\SH$ we have $\cat C_{p,n} \in \Supp(A)$ if and only if $K(p,n-1)_*(A)\neq 0$. The following is surely well-known:
\end{Rem}

\begin{Lem}\label{lem:finite-to-infinite}
	If $\cat C_{p,n} \in \Supp(A)$ for all finite $1 \le n < \infty$ then $\cat C_{p,\infty} \in \Supp(A)$.
\end{Lem}

\begin{proof}
	The claim is that if $A$ is a weak ring (with unit $\eta:\unit \to A$) then $\HFp \otimes A=0$ implies $K(n) \otimes A=0$ for some finite $n$. (We drop the $p$ from the notation.) Write $\HFp = \hocolim_n P(n)$ as in~\cite[1.2]{HopkinsSmith98} and let $u_n : \unit \to P(n)$ denote the unit map of the homotopy ring spectrum $P(n)$. Then $\HFp \otimes A =0$ implies $0 = \colim_n \pi_0(P(n) \otimes A)$. Thus $u_1 \otimes \eta:\unit \to P(1) \otimes A$ vanishes in this colimit. Since the maps $P(n) \to P(n+1)$ are ring homomorphisms, it follows that $u_n \otimes \eta = 0$ for some $n$, which implies that $P(n) \otimes A =0$. This in turn implies $K(n) \otimes A = 0$ by the equality of Bousfield classes $\langle P(n) \rangle = \langle K(n) \rangle \vee \langle P(n+1)\rangle$ established in~\cite[Theorem 2.1]{Ravenel84}.
\end{proof}

\begin{Rec}
	The $p$-local stable homotopy category $\SH_{(p)}$ is the finite localization of $\SH$ with respect to the Thomason subset $W\coloneqq \bigcup_{q \neq p} \overbar{\{\cat C_{q,2}\}}$. That is, $\SH_{(p)}$ is the restriction ${\SH}|_{W^{\complement}}$ to the (non-open) subset $W^{\complement} = \SET{\cat C_{p,n}}{1 \le n\le \infty}$. When speaking of the points $\cat C_{p,n} \in \Spc(\SH_{(p)}^{\,\dname})$ we drop the $p$ and just write $\cat C_n$.
\end{Rec}

\begin{Exa}
	Let $\cat T=\SH_{(p)}$ be the $p$-local stable homotopy category. Bousfield completion with respect to $Y \coloneqq \smash{\overbar{\{\cat C_2\}}}$ is called $p$-completion and $\TYpp \eqqcolon \SH_p^{\wedge}$ is the category of $p$-complete spectra. Note that $Y = \supp(M(p))$ where $M(p) = \cone(\smash{\unit \xrightarrow{p} \unit})$ is the mod-$p$ Moore spectrum.
\end{Exa}

\begin{Exa}
	Let $\cat T=\SH$ be the stable homotopy category and consider the Thomason subset $Y \coloneqq \overbar{\{\cat C_{p,2}\}} = \supp(M(p))$. Note that $Y \cup \{\cat C_1\}$ is the complement of the Thomason subset $W(Y) = \bigcup_{q	\neq p} \overbar{\{\cat C_{q,2}\}} = \bigcup_{q \neq p} \supp(M(q))$ from \cref{Rem:W(Y)}. Recall that the localization $\SH \to {\SH}|_{W(Y)^{\complement}} = \SH_{(p)}$ is precisely $p$-localization. By Tate excision (\cref{Thm:Tate-excision}) we recover the standard equivalence
	\[\begin{tikzcd}
		\SH \ar[d]\ar[r]& \SH_{(p)} \ar[d]\\
		\SH_{Y}^{\perp\perp} \ar[r,"\simeq"] & (\SH_{(p)})_{Y}^{\perp\perp}.
	\end{tikzcd}\]
	In other words, $p$-completion can be performed either on $\SH$ or on $\SH_{(p)}$ depending on convenience and personal taste.
\end{Exa}

\begin{Prop}[Strickland] \label{prop:p-complete-ess-surj}
	The $p$-completion $\SHd \to (\SH_p^{\wedge})^{\dname}$ is essentially surjective: Every dualizable $p$-complete spectrum is the $p$-completion of a finite spectrum.
\end{Prop}

\begin{proof}
	This is established in~\cite[Proposition 5.31]{Strickland24pp}. Note that although \emph{loc.~cit.}~ is stated for a \emph{bounded below} $p$-complete spectrum $X$, the implication $(c)\Rightarrow (f)$ holds for any $p$-complete spectrum, since a dualizable $p$-complete spectrum is necessarily bounded below. Indeed, the homotopy groups $\pi_i(X)$ of a $p$-complete spectrum are Ext-$p$-complete; see~\cite[Section 5]{Strickland24pp}. If $X \in \SH_p^{\wedge}$ is dualizable then $X \otimes M(p) = X/p$ is compact. Hence $\pi_i(X/p)=0$ for $i \ll 0$, which implies $p:\smash{\pi_i(X) \xrightarrow{\sim} \pi_i(X)}$ is an isomorphism for $i \ll 0$. This implies that $X$ is bounded below since an abelian group $A$ which is Ext-$p$-complete must vanish if $p:A\to A$ is an isomorphism; see~\cite[Section 12]{Strickland20pp}.
\end{proof}

\begin{Prop}\label{prop:surjective-embedding}
	Let $F:\cat K \to \cat L$ be a tt-functor which is essentially surjective. Then $\Spc(F):\Spc(\cat L)\to \Spc(\cat K)$ is a topological embedding.
\end{Prop}

\begin{proof}
	The map $\varphi\coloneqq \Spc(F)$ is injective by~\cite[Corollary 3.8]{Balmer05a}. Let $Z \subseteq \Spc(\cat L)$ be a Thomason closed subset. Then $Z=\supp(a)$ for some $a \in \cat L$ and $a \simeq F(b)$ or some $b \in \cat K$. Hence $Z=\supp(F(b)) = \varphi^{-1}(\supp(b))$. Thus $\varphi(Z) = \varphi(\varphi^{-1}(\supp(b))) = \Img \varphi \cap \supp(b)$. Thus $\varphi:\Spc(\cat K) \to \Img \varphi$ is a spectral map which sends Thomason closed subsets to closed sets. By~\cite[Theorem 5.3.3]{DickmannSchwartzTressl19}, this implies that $\varphi$ is a closed map. Hence $\varphi$ is a homeomorphism onto its image.
\end{proof}

\begin{Thm}\label{Thm:p-completion}
	The $p$-completion $\SH_{(p)} \to \SH_p^{\wedge}$ induces a homeomorphism
	\[
		\varphi : \Spc( (\SH_p^{\wedge})^{\dname}) \xrightarrow{\sim} \Spc(\SHdp).
	\]
\end{Thm}

\begin{proof}
	We know that $\Img \varphi \supseteq Y$ by \cref{Thm:Spc(compl)-on-Y}. Moreover, the generic point $\cat C_1$ is also contained in $\Img \varphi$ by \cref{Cor:immediate-generalization} since it is an immediate generalization of $\cat C_2 \in Y$. Therefore $\varphi$ is surjective. On the other hand, the functor $\SHdp \to (\SH_p^{\wedge})^{\dname}$ is essentially surjective by \cref{prop:p-complete-ess-surj}. Invoking \cref{prop:surjective-embedding}, we conclude that $\varphi$ is a surjective embedding.
\end{proof}

\begin{Rem}
	The usefulness of the Tate Intermediate Value Theorem is potentially limited by the possibility that $\Supp(\tY(\unit))$ could simply consist of all generalizations of~$Y$, that is, $\Supp(\tY(\unit)) = \gen(Y) \cap Y^{\complement}$. In this case, the theorem is not particularly helpful. For example, if $\cat T$ is local with Thomason closed point $Y=\{\mathfrak m\}$, then it could be the case that the map induced by completion $\Spc(\smash{\hTd})\to\Spc(\Td)$ is surjective. This is to be expected in noetherian commutative algebra, but \cref{exa:valuation-domain} shows that it need not hold in the non-noetherian case. Let us give an example of this phenomenon in (non-noetherian) chromatic homotopy theory.
\end{Rem}

\begin{Exa}
	Let $\cat T=\SH_{E(n)}$ denote the category of $E(n)$-local spectra (at some implicit prime $p$). The spectrum
        \[
			\Spc(\SHEnd) = \{ \cat C_{1} \leadsto \cdots \leadsto \cat C_{n}\leadsto \cat C_{n+1}\}
        \]
	is a local irreducible space consisting of $n+1$ points, as depicted above. Here $\cat C_h = \SET{ x \in \SHEnd}{K(h-1)_*(x)=0}$. This result is due to Hovey--Strickland~\cite{HoveyStrickland99} and is also discussed in~\cite[Section~10]{BarthelHeardSanders23a}.
\end{Exa}

\begin{Prop}\label{Prop:Kn}
	The Bousfield completion of $\cat T=\SH_{E(n)}$ at the unique closed point can be identified with the $K(n)$-localization
		\[
			\SH_{E(n)} \to \SH_{K(n)}
		\]
	which induces a surjective map on spectra
	\begin{equation}\label{eq:Spc-Kn-En}
			\Spc(\SHKnd) \twoheadrightarrow \Spc(\SHEnd).
	\end{equation}
\end{Prop}

\begin{proof}
	Let $Y=\{ \cat C_{n+1}\}$ denote the unique closed point. Since the $E(n)$-local category $\cat T$ is stratified~\cite[Theorem 10.14]{BarthelHeardSanders23a} and hence satisfies the telescope conjecture~\cite[Theorem 9.11]{BarthelHeardSanders23a}, the finite localization on $\cat T$ associated to $Y$ coincides with the smashing localization $L_{n-1}$ associated to $E(n-1)$. It follows that, in the terminology of~\cite[Section~6.3]{HoveyStrickland99}, the Bousfield completion $\eY \otimes \cat T \cong [\eY,\cat T]$ is the monochromatic category $\cat M_n$ which~\cite[Theorem 6.19]{HoveyStrickland99} identifies with the $K(n)$-local category $\SH_{K(n)}$. It then follows from \cref{cor:img-of-compl} that the image of the map \eqref{eq:Spc-Kn-En} is precisely $\Supp(L_{K(n)}\unit)$, the support of the $K(n)$-local sphere.

	Recall from~\cite[Proposition 5.3]{HoveyStrickland99} that we have equalities of Bousfield classes
		\[
			\langle L_{K(n)} \unit \rangle = \langle L_{E(n)} \unit \rangle = \langle E(n) \rangle = \langle K(0) \rangle \vee \cdots \vee \langle K(n) \rangle.
		\]
	In other words, $L_{K(n)}\unit \otimes t = 0$ if and only if $K(i) \otimes t = 0$ for each $i=0,\ldots, n$. Taking $t$ to be the Morava $K$-theory $K(m)$ we see that $L_{K(n)}\unit \otimes K(m) = 0$ if and only if $m>n$. The homological support is given by tensoring with the Morava $K$-theories (see e.g.~\cite[(5.12)]{BarthelHeardSanders23b}). Hence (taking into account a shift by one) $\Supp(L_{K(n)}\unit) = \{ \cat C_1,\ldots,\cat C_{n+1}\}$, that is, the whole of $\Spc(\SHEnd)$.
\end{proof}

\begin{Rem}
	The surjectivity of the map \eqref{eq:Spc-Kn-En} was first established by Barthel, Heard and Naumann in their work on the Hovey--Strickland Conjecture, which asserts that the map is in fact a homeomorphism. The main theorem of~\cite{BarthelHeardNaumann22} is a verification of the height $n=2$ case; see \cite[Theorem~A and Remark~3.6]{BarthelHeardNaumann22}.
\end{Rem}

\begin{Rem}
	The above proposition shows that in the case of the $E(n)$-local category, the support of the associated Tate ring (for the unique closed point) is the whole spectrum (minus the closed point). Nevertheless, completing at the closed point of a local category \emph{can} result in a Tate ring with an interesting support, as we saw already in the minimal \Cref{exa:valuation-domain}.
\end{Rem}

\subsection{Examples from equivariant homotopy theory}

\begin{Exa}\label{exa:SH(G)}
	Let $G$ be a finite group and let $\SH(G)=\Ho(\Sp_G)$ denote the equivariant stable homotopy category. As a set, the spectrum $\Spc(\SHGd)$ consists of a number of copies of $\Spc(\SHd)$, one copy, or `layer'
	\[
		L_H \coloneqq \SET{\cat P_G(H,p,n)}{\text{all } p, 1 \le n \le \infty}
	\]
	for each conjugacy class of subgroups $H\le G$; see~\cite{BalmerSanders17}. The layer for the trivial subgroup $L_{\{e\}} = \supp(G_+) \eqqcolon Y$ is Thomason closed and its associated idempotent triangle is the isotropy separation sequence $EG_+ \to S^0 \to \smash{\widetilde{E}G}$. The Bousfield completion of $\SH(G)$ along $Y$ may be identified with the tt-category of Borel equivariant spectra $\SH(G)_{\mathrm{borel}} \coloneqq \Ho(\Fun(BG,\Sp))$; see~\cite[Proposition~6.17]{MathewNaumannNoel17}. The $Y$-complete $G$-spectra $\SH(G)_Y^{\perp\perp}=[\eY,\SH(G)] \subset \SH(G)$ are often called the \emph{Borel complete} $G$-spectra.
\end{Exa}

\begin{Not}
	We write $\tG:\SH(G)\to \SH(G)$ for the Tate construction $\tY$ associated to $Y=\supp(G_+)$. That is, $\tG(-) = \widetilde{E}G \otimes [EG_+,-]$.
\end{Not}

\begin{Not}
	Given a nonequivariant spectrum $s \in \SH$, we write $\bG(s)\coloneqq [EG_+,\triv_G(s)]$ for the \emph{Borel equivarization} (or \emph{Borel completion}) of $s$.
\end{Not}

\begin{Rem}\label{rem:equivariant-norm}
	The `norm' exact triangle \eqref{eq:norm-exact-triangle} takes the form
	\[
		EG_+ \otimes t \to [EG_+,t] \to \tG(t) \to \Sigma EG_+ \otimes t
	\]
	for any $t \in \SH(G)$. Applying the fixed point functor $(-)^G : \SH(G) \to \SH$ and using the Adams isomorphism one obtains the norm cofiber sequence
	\[
		t_{hG} \to t^{hG} \to \tG(t)^G \to \Sigma t_{hG}.
	\]
\end{Rem}

\begin{Rem}
	The category $\SH(G)$ satisfies the steel condition and for any weak ring $A$ we have $\cat P(H,p,n) \in \Supp(A)$ if and only if $\cat C_{p,n} \in \Supp(\Phi^H(A))$ if and only if $K(p,n-1)_*(\Phi^H(A)) \neq 0$; cf.~\cref{rem:supp-in-SH}.
\end{Rem}

\begin{Thm}\label{Thm:equiv}
	Let $G$ be a finite abelian group and let $\cat T=\SH(G)_{(p)}$ be the $p$-localization of $\SH(G)$. For any $H \le G$, we have
	\[
		\Supp(\tG(\unit)) \cap L_H = \begin{cases}
			L_H & \text{if $H$ is a nontrivial $p$-group} \\
			\emptyset & \text{otherwise}.
			\end{cases}
	\]
	That is, the support of the Tate ring consists precisely of the layers corresponding to the nontrivial $p$-subgroups.
\end{Thm}

\begin{proof}
	For any subgroup $H \le G$, the restriction functor $\SH(G)_{(p)} \to \SH(H)_{(p)}$ satisfies Grothendieck--Neeman duality and hence is a closed functor (\cref{rem:closed-functor}). Hence, if $\varphi_H^G : \Spc(\SH(H)_{(p)}\cc)\to\Spc(\SH(G)_{(p)}\cc)$ denotes the induced map, then it follows from \cref{Cor:tate-supp-base-change}(b) that $\Supp(\tH(\unit)) = (\varphi_H^G)^{-1}(\Supp(\tG(\unit))$. Thus, $\cat P_G(H,p,n) \in \Supp(\tG(\unit))$ if and only if $\cat P_H(H,p,n) \in \Supp(\tH(\unit))$. We are thus reduced to proving the $H=G$ case of the theorem.

	The statement is true if $G=1$ is the trivial group, since $Y=L_H$ in that case. Note that if $G\neq 1$ then $\Phi^G(\tG(\unit)) = \Phi^G(\bG(\unit))$ since the trivial family of subgroups is then contained in the family of proper subgroups. By~\cite[Theorem 4.25]{MathewNaumannNoel19}, the derived defect base of the Borel equivarization of the $p$-local sphere $\bG(\unit)$ is the family of $p$-subgroups of $G$. In other words, $\bG(\unit) \in \Thick_{\otimes}\langle G/K_+ \mid K \text{ is a $p$-subgroup}\rangle$. If $G$ is \emph{not} a $p$-group, then this family is contained in the family of proper subgroups, and hence $\Phi^G(\bG(\unit))=0$. This establishes that $\Supp(\tG(\unit)) \cap L_G = \emptyset$ if $G$ is not a nontrivial $p$-group.
	
	Finally, suppose that $G$ is a nontrivial (abelian) $p$-group. For any $n$, let $E_n$ denote a Lubin--Tate spectrum at $p$ of height $n$. The ring homomorphism $S^0 \to E_n$ induces a ring homomorphism $\bG(\unit)\to \bG(E_n)$. This in turn induces a ring homomorphism $\Phi^G(\tG(\unit))=\Phi^G(\bG(\unit))\to \Phi^G(\bG(E_n))$. Thus $\Supp(\Phi^G(\tG(\unit))) \supseteq \Phi^G(\Supp(\bG(E_n)))$ by \cref{rem:unital-map}. A key result~\cite[Theorem 3.5]{BHNNNS19} establishes that if $n \ge \rank_p(G)$ then the height of $\Phi^G(\bG(E_n))$ is $n-\rank_p(G)$. This implies that $\cat C_{1+n-\rank_p(G)} \in \Supp(\Phi^G(\bG(E_n))$; see~\cite[Remark 3.2]{BHNNNS19} and~\cite[Remark 10.29]{excisive}. Hence $\cat C_{1+n-\rank_p(G)} \in \Supp(\Phi^G(\tG(\unit)))$. Thus, letting $n$ vary, we see that $\cat C_{m} \in \Supp(\Phi^G(\tG(\unit)))$ for all finite $m$. It then follows that $\cat C_{\infty} \in \Supp(\Phi^G(\tG(\unit)))$ by \cref{lem:finite-to-infinite}. This establishes that $\Supp(\tG(\unit)) \cap L_G = L_G$ when $G$ is a nontrivial $p$-group, which completes the proof.
\end{proof}

\subsection{Examples from modular representation theory}

%
\begin{Exa}\label{exa:KInj}
	Let $G$ be a finite group and $k$ a field of characteristic $p>0$. Let $\cat T=\DRep(G;k)\simeq \KInj(kG)$ be the category of derived $k$-linear $G$-representations. Recall from \cref{Exa:KInj} that the Bousfield completion $\cat T\to\TYpp$ at the unique closed point $Y=\{\mathfrak m\}$ is the canonical functor $\KInj(kG) \twoheadrightarrow \Der(kG)$, and that this induces an equivalence $\KInj(kG)\cc \xrightarrow{\sim} \Db(kG\mmod) = \Der(kG)\dd$.
\end{Exa}

\begin{Prop}\label{prop:[e1]-modules}
	Assume $\cat T=\Ho(\cat C)$ has an underlying model and let
	\[
		f^*:\cat T =\Ho(\cat C)\to\Ho([\eY,\unit]\text{-}\mathrm{Mod}_{\cat C})\eqqcolon \cat S
	\]
	be base-change along $\unit \to [\eY,\unit]$. The Bousfield completion $\cat T \to \TYpp$ factors as
	\begin{equation*}
		\cat T \to \cat S \to \TYpp
	\end{equation*}
	in which the tt-functor $\cat S \to \TYpp$ is fully faithful on small objects: $\cat \Sd \hookrightarrow \hTd$.
\end{Prop}

\begin{proof}
	Recall that $\cat S$ is rigidly-compactly generated by $f^*(\cT\cc)$. By construction, the unit $\unit_{\cat S}$ is $f^{-1}(Y)$-complete. Indeed, applying the conservative functor $f_*$ to the canonical map $\unit_{\cat S} \to [\mathbb{e}_{f^{-1}(Y)},\unit_{\cat S}]$, we obtain $f_*(\unit_{\cat S}) \to f_*[\mathbb{e}_{f^{-1}(Y)},\unit_{\cat S}] \simeq f_*[f^*(\eY),\unit_{\cat S}] \simeq [\eY,f_*(\unit_{\cat S})]$ which is an isomorphism since $f_*(\unit_\cat S)=[\eY,\unit_{\cat T}]$ is $Y$-complete. It then follows from~\cite[Theorem 2.16]{BalmerSanders26} that the bottom tt-functor in the commutative diagram
	\[\begin{tikzcd}
		\cat T \ar[r,"f^*"] \ar[d] & \cat S \ar[d] \\
		\TYpp \ar[r,"\simeq"] & \cat S_{f^{-1}(Y)}^{\perp\perp}
	\end{tikzcd}\]
	of \cref{Prop:functorial}(d) is an equivalence. On the other hand, since $\unitS$ is $f^{-1}(Y)$-complete, the right arrow is fully faithful on $\Sc=\Sd$.
\end{proof}

\begin{Prop}\label{Prop:DMack}
	Let $G$ be a finite $p$-group and let $k$ be a field of characteristic~$p$. Let $\cat T=\DMack(G;k)$ be the category of derived Mackey functors. Its spectrum $\Spc(\Td)$ is the lattice of conjugacy classes of subgroups of $G$ with the trivial subgroup providing a unique closed point $\mathfrak m$. The completion $\Td \to \hTd$ at $Y=\{\mathfrak m\}$ may be identified with the canonical functor $\DMack(G;k)\cc \to \DRep(G;k)\cc$.
\end{Prop}

\begin{proof}
	The spectrum $\Spc(\Td)$ is computed in~\cite{PatchkoriaSandersWimmer22}; see also~\cite[Part~V]{BarthelHeardSanders23a}. We are completing along $Y=\supp(G_+)$. According to \cref{prop:[e1]-modules}, the functor $\cat T \to \TYpp$ factors as $\cat T \to \cat S \to \TYpp$. Regarding $\cat T\simeq\Ho(\triv_G(\rmH k)\text{-Mod}_{\Sp_G})$ as modules over the equivariant ring spectrum $\triv_G(\Hk) \in \CAlg(\Sp_G)$, the map of rings $\unit \to [\eY,\unit]$ amounts to
		\[
			\triv_G(\Hk) \to [EG_+,\triv_G(\Hk)] \eqqcolon \bG(\Hk).
		\]
	As explained in~\cite[Theorem 3.7]{barthel2021rep2}, we have an equivalence of tt-categories $\Ho(\bG(\Hk)\text{-Mod}_{\Sp_G}) \simeq \DRep(G;k)$. The functor $\cat T \to \cat S$ may thus be identified with the geometric functor $\DMack(G;k)\to\DRep(G;k)$. Moreover, on dualizable objects, $\Sd \hookrightarrow \smash{\hTd}$ is an equivalence, since $\cat S \cong \DRep(G;k)$ has the property that $\Sc=\Sd \to (\cat S_{f^{-1}(Y)}^{\perp\perp})^{\dname}$ is an equivalence (\cref{exa:KInj}).
\end{proof}

\begin{Exa}
	Let $G=D_8$ be the dihedral group of order $8$ and let $k=\overbar{\bbF_{\!2}}$.
	
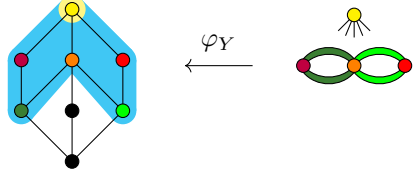
\[
	\begin{tikzpicture}[scale=1.5]
  \begin{scope} [rotate=-90, yshift=2cm]
	\def\diabulletsize{0.06}
	\def\diaxoffset{0.45}
	\def\diayoffset{0.45}
	\coordinate (b) at (0,0);
	\coordinate (y) at ($(b)-(1*\diaxoffset,0)$);
	\draw (y) -- ++(0:0.2);
	\draw (y) --  ++(45:0.2);
	\draw (y) --  ++(22.5:0.2);
	\draw (y) --  ++(-22.5:0.2);
	\draw (y) --  ++(-45:0.2);
	\filldraw[fill=yellow,draw=black] ($(b)-(1*\diaxoffset,0)$) circle (\diabulletsize);
	\draw[fill=green,draw=black] ($(b)+(0.06,0)$) to [bend right=50] ($(b)+(0.06,\diayoffset)$) to ($(b)+(0,\diayoffset)$) to [bend left=50] ($(b)$);
	\draw[fill=green,draw=black] ($(b)+(-0.06,0)$) to [bend left=50] ($(b)+(-0.06,\diayoffset)$) to ($(b)+(0,\diayoffset)$) to [bend right=50] ($(b)$);
	\draw[fill=OliveGreen,draw=black] ($(b)+(0.06,0)$) to [bend left=50] ($(b)+(0.06,-\diayoffset)$) to ($(b)+(0,-\diayoffset)$) to [bend right=50] ($(b)$);
	\draw[fill=OliveGreen,draw=black] ($(b)+(-0.06,0)$) to [bend right=50] ($(b)+(-0.06,-\diayoffset)$) to ($(b)+(0,-\diayoffset)$) to [bend left=50] ($(b)$);
	\filldraw[fill=purple,draw=black] ($(b)+(0,-\diayoffset)$) circle (\diabulletsize);
	\filldraw[fill=red,draw=black] ($(b)+(0,\diayoffset)$) circle (\diabulletsize);
	\filldraw[fill=orange,draw=black] (b) circle (\diabulletsize);
  \end{scope}
  \draw[<-] (0.5,0) -- (1.1,0); \node at (0.8,0.2) {$\varphi_Y$};
  \begin{scope}[rotate=-90, xshift=-0.5cm, yshift=-0.5cm]
	\def\diabulletsize{0.06cm}
	\def\diaxoffset{0.45}
	\def\diayoffset{0.45}
	\coordinate (b) at (0,0);
	\fill[fill=cyan!60] ($(b)+(\diaxoffset,\diayoffset)$) circle (2*\diabulletsize);
	\fill[fill=cyan!60] ($(b)+(2*\diaxoffset,\diayoffset)$) circle (2*\diabulletsize);
	\fill[fill=cyan!60] ($(b)+(\diaxoffset,0)$) circle (2*\diabulletsize);
	\fill[fill=cyan!60] ($(b)+(2*\diaxoffset,-\diayoffset)$) circle (2*\diabulletsize);
	\coordinate (s) at (b);
	\fill[fill=cyan!60] (s) ++(135:2*\diabulletsize) -- ++ (\diaxoffset,\diayoffset) -- ($(s) + (\diaxoffset,\diayoffset) + (-45:2*\diabulletsize)$)
		-- ($(s) + (2*\diabulletsize,0)$)
		-- ($(s) + (\diaxoffset,-\diayoffset) + (45:2*\diabulletsize)$)
		-- ($(s)+(\diaxoffset,-\diayoffset)+(-135:2*\diabulletsize)$)
		-- ($(s)+(-135:2*\diabulletsize)$);
	\coordinate (s) at (b);
	\fill[fill=cyan!60] ($(s)+(\diaxoffset,-\diayoffset)$) circle (2*\diabulletsize);
	\fill[fill=cyan!60] (s) ++(135:2*\diabulletsize) -- ++ (\diaxoffset,\diayoffset) -- ($(s) + (\diaxoffset,\diayoffset) + (-45:2*\diabulletsize)$)
		-- ($(s) +(-45:2*\diabulletsize)$);
	\coordinate (s) at ($(b)+(\diaxoffset,0)$);
	\fill[fill=cyan!60] (s) ++(135:2*\diabulletsize) -- ++ (\diaxoffset,\diayoffset) -- ($(s) + (\diaxoffset,\diayoffset) + (-45:2*\diabulletsize)$)
		-- ($(s) + (2*\diabulletsize,0)$)
		-- ($(s) + (\diaxoffset,-\diayoffset) + (45:2*\diabulletsize)$)
		-- ($(s)+(\diaxoffset,-\diayoffset)+(-135:2*\diabulletsize)$)
		-- ($(s)+(-135:2*\diabulletsize)$);
	\coordinate (s) at (b);
	\fill[fill=cyan!60] ($(s)+(\diaxoffset,\diayoffset)+(0,2*\diabulletsize)$) -- ++(\diaxoffset,0) -- ++(0,-2*\diabulletsize)
		-- ($(s)+(\diaxoffset,0)+(2*\diabulletsize,0)$)
		-- ($(s)+(2*\diaxoffset,-\diayoffset)$)
		-- ++(0,-2*\diabulletsize)
	-- ++(-\diaxoffset,0)
	-- ++(0,2*\diabulletsize)
	-- (s);
	\fill[fill=yellow!60] (b) circle (2*\diabulletsize);
	\draw (b) -- ($(b) + (\diaxoffset,\diayoffset)$) --
	($(b) + (2*\diaxoffset,\diayoffset)$) --
	($(b) + (3*\diaxoffset,0)$);
	\draw (b) -- ($(b) + (\diaxoffset,-\diayoffset)$) --
	($(b) + (2*\diaxoffset,-\diayoffset)$) --
	($(b) + (3*\diaxoffset,0)$);
	\draw ($(b) +(\diaxoffset,0)$) -- ($(b) + (2*\diaxoffset,\diayoffset)$);
	\draw ($(b) +(\diaxoffset,0)$) -- ($(b) + (2*\diaxoffset,-\diayoffset)$);
	\draw (b) -- ($(b) + (3*\diaxoffset,0)$);
	\filldraw[fill=yellow] (b) circle (\diabulletsize);
	\filldraw[fill=orange] ($(b)+(\diaxoffset,0)$) circle (\diabulletsize);
	\filldraw ($(b)+(2*\diaxoffset,0)$) circle (\diabulletsize);
	\filldraw ($(b)+(3*\diaxoffset,0)$) circle (\diabulletsize);
	\filldraw[fill=red] ($(b)+(\diaxoffset,\diayoffset)$) circle (\diabulletsize);
	\filldraw[fill=green] ($(b)+(2*\diaxoffset,\diayoffset)$) circle (\diabulletsize);
	\filldraw[fill=purple] ($(b)+(\diaxoffset,-\diayoffset)$) circle (\diabulletsize);
	\filldraw[fill=OliveGreen] ($(b)+(2*\diaxoffset,-\diayoffset)$) circle (\diabulletsize);
  \end{scope}
	\end{tikzpicture}
	\]

\vspace{-0.5em}

\noindent The spectrum of~$\DMack(G;k)\cc$ is the lattice of conjugacy classes of subgroups of~$D_8$; it has 8 points as depicted on the left-hand side above, with specialization going upwards, from larger subgroups to smaller ones. Each of these conjugacy classes is a singleton, except the two red points which each represents a pair of $G$-conjugate cyclic~$C_2$'s. Our chosen Thomason subset~$Y$ is the yellow point corresponding to the trivial subgroup. The two green points correspond to the two Klein-four subgroups of~$D_8$, which are its maximal elementary abelian $2$-subgroups. Finally, it will follow from our description of~$\varphi_Y$ that the cyan region is the support of the Tate ring.

	On the right-hand side, the spectrum of the $Y$-completion $\DMack(G;k)\dd_{\mathfrak m}\cong\Db(kG\mmod)$ is~$\Spech(\rmH^\sbull(D_8,k))$. By Quillen Stratification, the latter has two irreducible components (the closures of the green curves), one for each Klein-four in~$D_8$. More precisely, we are mapping the spectra of the derived categories of elementary abelian subgroups to the spectrum of~$\Db(kG\mmod)$ via the map induced on $\Spc(-)$ by the restriction functor. Each of these two irreducible components is the image of $\Spech(\rmH^\sbull(C_2^{\times2},k))$, that is, a $\bbP^1_k$ with a closed point on top. Two of the $\bbF_{\!2}$-rational points of~$\bbP^1_k$ get identified over~$G$, because each Klein-four has two cyclic subgroups that get conjugated in~$D_8$. This phenomenon is known as fusion. This gives the (outer) red point on each component. (The orange point comes from the center $Z(D_8)\cong C_2$ which does not get `fused' with any other $C_2$ in~$D_8$.) These two components are glued together along the Sierpi\'nski space coming from the center of~$D_8$ (the middle orange point and the yellow point). In particular, the right-hand spectrum has again a unique closed point (the irrelevant ideal~$\rmH^+(G,k)$ in~$\Spech(\rmH^\sbull(G,k))$); this is the yellow point in the right-hand side.

	The map $\varphi_Y\colon\Spc(\hTd)\to \SpcT$ is described by the colors: The preimage of each point of the left-hand lattice is the part with the same color in the right-hand picture. The yellow point is the unique preimage of the yellow point (compare \Cref{Thm:Spc(compl)-on-Y}), each green area goes to the corresponding green point, etc. This calculation is an easy verification from the following three facts. First, the object $F_G(G/H_+)$ in~$\DMack(G,k)^c$ is mapped to~$k(G/H)$ in~$\Db(kG\mmod)$.
	Second, the support of $F_G(G/H_+)\cong\Ind_H^G(\unit)$ on the left-hand side consists of the image under $\Spc(-)$ of restriction to~$H$, by general \'etale tt-geometry. See~\cite{Balmer16b}. This amounts to the conjugacy classes of subgroups of~$H$. Similarly, by the same generalities, the support of~$k(G/H)\cong\Ind_H^G(\unit)$ on the right-hand side consists of the image under $\Spc(-)$ of restriction to~$H$. These subsets are precisely the ones described above, in Quillen Stratification. And thirdly, again by general tt-geometry, the preimage under~$\varphi=\Spc(F)$ of the support $\supp(c)$ of an object is the support of the image~$F(c)$ of said object, here applied to~$c=F_G(G/H_+)$. The result follows by inspection, by varying the subgroup~$H$ among the elementary abelian subgroups of~$G$, starting with~$H=1$ and then increasing the $2$-rank one by one.
\end{Exa}

\subsection{Excisive examples}

We will now illustrate completion in the case of categories of excisive functors. This is based on~\cite{excisive} and we will assume some familiarity with that work.

\begin{Not}
	We fix a prime $p$ and work $p$-locally throughout. For each $d \ge 1$, we have the tt-category of $p$-local $d$-excisive functors
	\[
		\Exc{d} \coloneqq \Exc{d}(\Sp\cc,\Sp)_{(p)}\simeq \Exc{d}(\Sp\cc,\Sp_{(p)}).
	\]
\end{Not}

\begin{Rec}
	As a set, the spectrum $\Spc(\Exc{d}^{\dname})$ consists of $d$ disjoint copies of $\Spc(\Sp_{(p)}^{\dname})$ which are pulled back via the Tate derivatives
	\[
		\partial_\ell : \Exc{d} \to \Sp_{(p)}, \qquad 1 \le \ell \le d.
	\]
	The prime $\Spc(\partial_\ell)(\cat C_h)$ is denoted $\cat P_d([\ell],h)$ and we write $L_{\ell} \coloneqq \Img(\Spc(\partial_\ell)) = \SET{\cat P_d([\ell],h)}{1 \le h \le \infty}$ for the $\ell$th layer.
\end{Rec}

\begin{Rem}
	As explained in~\cite[Corollary 7.21]{excisive}, the steel condition holds for $\Exc{d}$ and the homological residue field at $\cat P_d([\ell],h)$ is the functor $K(h-1)_* \circ \partial_\ell$. Thus, for a weak ring $A$, we have
	\begin{align*}
		\cat P_d(\num{\ell},h) \in \Supp(A) &\Longleftrightarrow K(h-1)_* \partial_\ell(A) \neq 0 \\
										 &\Longleftrightarrow \cat C_h \in \Supp(\partial_\ell(A))
	\end{align*}
	by invoking~\cite[Theorem 1.8]{Balmer20_bigsupport} and \cref{rem:supp-in-SH}.
\end{Rem}

\begin{Rem}
	The category $\Exc{d}$ is rigidly-compactly generated by a set of compact-dualizable generators $P_d h_{\Sphere}(i)$, $1 \le i \le d$, which includes the unit $\unit = P_d h_{\Sphere}(1)$. The finite localization associated to the Thomason closed set $Y\coloneqq \supp(P_d h_{\Sphere}(d)) = L_d$ may be identified with the $(d-1)$-excisive approximation $P_{d-1} : \Exc{d} \to \Exc{d-1}$. Moreover, note that $Y=L_d$ is a copy of $\Spc(\Sp_{(p)}\dd)$. In~\cite{excisive}, the spectrum $\smash{\Spc(\Exc{d}^{\dname})}$ is computed by induction on $d$ utilizing this open-closed decomposition. It provides an excellent illustration of the philosophy of the present paper. Set-theoretically we have $\Spc(\Exc{d}^{\dname}) = Y \sqcup Y^{\complement} = \Spc(\Sp_{(p)}^{\dname}) \sqcup \Spc(\Exc{d-1}^{\dname})$. The key to the inductive step is to understand how these pieces are glued together.
	\[\input{exc-open-closed-tikz.tex}\]
\end{Rem}

\vspace{-1em}

\begin{Rem}
	The Tate construction $\tY:\Exc{d} \to \Exc{d}$ of interest to us is denoted~$t_d$ in~\cite{excisive}. We write $\td$ for it.
\end{Rem}

\begin{Rem}
	Recall from~\cite[Definition 10.32]{excisive} that a $p$-power partition of~$d$ of length $\ell$ is a partition $\lambda = (d_1,\ldots,d_\ell) \vdash d$ such that each $d_i$ is a power of $p$. (Here we include $p^0=1$ as a power of $p$.)
\end{Rem}

\begin{Thm}\label{Thm:excisive}
	Let $d \ge 1$ be an integer and let $\cat T=\Exc{d}(\Sp\cc,\Sp)_{(p)}$. For any $1 \le \ell \le d$, we have
		\[
			\Supp(\td(\unit)) \cap L_\ell = \begin{cases}
			L_\ell & \text{if $d > \ell$ and $d$ admits a $p$-power partition of length $\ell$}\\
			\emptyset & \text{otherwise}
			\end{cases}
		\]
\end{Thm}

\begin{proof}
	Let $i_d:\Sp_{(p)} \to \Exc{d}$ denote the canonical geometric functor and consider $\smash{i_d(L_{p,h}^f\Sphere)} \in \Exc{d}$. The key theorem~\cite[Theorem 10.33]{excisive} establishes that if $d> \ell$ and $d$ admits a $p$-power partition of length $\ell$, then
		\[
			\height_p(\partial_\ell\, \td(i_d L_{p,h}^f\Sphere)) = h-1
		\]
	which implies that $\cat C_h \in \Supp(\partial_\ell\, \td(i_d L_{p,h}^f\Sphere))$. This implies $\cat C_h \in \Supp(\partial_\ell\, \td(\unit))$ by \cref{rem:unital-map} due to the ring map
		\[
			\unit=i_d \Sphere \to i_d L_{p,h}^f \Sphere.
		\]
 	Since this is true for all finite $1\le h <\infty$, it then follows from \cref{lem:finite-to-infinite} that $\cat C_\infty \in \Supp(\partial_\ell\, \td(\unit))$ too. This proves that, in this case, $\Supp(\td(\unit)) \cap L_\ell = L_\ell$.

	It remains to prove that $\partial_\ell\,\td(\unit) = 0$ in the contrary case. This is the $h=\infty$ version of the vanishing result of~\cite[Theorem 10.33]{excisive} interpreting $L_{p,\infty}^f$ as the $p$-local sphere. One readily verifies that the argument goes through with this interpretation.
\end{proof}

\begin{Exa}
	For $d=3$ and the prime $p=2$, the spectrum is the following picture

	\medskip
	\begin{center}
		 \input{exc3-p2-tikz.tex}
	\end{center}

	\noindent
	where $Y=\supp(P_3 h_{\Sphere}(3))$ is the yellow region (the third layer) and $\Supp(\tY(\unit))$ is precisely the cyan region (the second layer). Observe how all inclusions between the third and first layer are mediated through the second, as predicted by \cref{Thm:TIV}.
\end{Exa}

\begin{Rem}
	The Bousfield completion of $\Exc{d}$ along $Y$ may be identified with the tt-category of Borel equivariant $\Sigma_d$-spectra; see~\cite[Proposition 6.23]{excisive}. Note that this is the same as the Bousfield completion of $\SH(\Sigma_d)$ considered in \cref{exa:SH(G)}. The tt-geometry of $\SH(\Sigma_d)_{\mathrm{borel}}$ is not well-understood. Nevertheless, in order to compute the inclusions between $Y$ and $\Supp(\tY(\unit))$, one does not always need to have complete understanding of $\smash{\Spc(\hTd)}$. Indeed, in the $d=3$, $p=2$ example above, the left-hand inclusions between $Y=L_3$ and $\Supp(\tY(\unit))=L_2$ can be completely determined by Tate support computations. In more detail, for each $h$ we may consider the truncated category $\cat T_{\le h} = i_d(L^f_{h-1}) \otimes \cat T$ in the sense of~\cite[Example 6.15]{excisive}, which is just the restriction $\cat T\to \cat T|_{U_{\le h}}$ to the open subset
	\[
		U_{\le h} = \SET{\cat P_d([\ell],n)}{1\le \ell\le d,1\le n\le h} \subseteq \Spc(\cat T^{\dname}).
	\]
	By \cref{Cor:tate-supp-base-change}(a), we have $\Supp(\mathbb{t}_{Y \cap U_{\le h}}(\unit))=\Supp(\td(i_d(L^f_{h-1})))$. As discussed above, $\Supp(\mathbb{t}_3(i_3(\smash{L^f_{h-1}}))) = \SET{\cat P_3([2],n)}{1\le n \le h-1}$. We leave it as an exercise for the interested reader to see how these computations suffice to determine the inclusions between layer 3 and layer 2 above; the argument in~\cite[Lemma~11.15]{excisive} may be helpful. In summary, this provides a proof of concept where all the gluing information between $Y$ and $Y^{\complement}$ can be obtained by computing the supports of Tate rings. A more general understanding of this phenomenon would be desirable.
\end{Rem}


\newcommand{\etalchar}[1]{$^{#1}$}

\end{document}

%% file: exc-open-closed-tikz.tex
\def\diabulletsize{0.06cm}
\def\diayoffset{0.45}
\def\diaxoffset{0.45}
\def\dianumbasenodes{4}
\def\dianumbasenodesminus{3}
\def\dianumbasenodesminust{2}
\begin{tikzpicture}
	\coordinate (base) at (0,0);


	\coordinate (start) at (base);
	\filldraw[color=yellow!60] ($(start)+(-2*\diabulletsize,0)$) -- ($(start)+(-2*\diabulletsize,\dianumbasenodes*\diayoffset+1.25*\diayoffset)$)
	-- ($(start)+(2*\diabulletsize,\dianumbasenodes*\diayoffset+1.25*\diayoffset)$)
	-- ($(start)+(2*\diabulletsize,0)$) -- ($(start)+(-2*\diabulletsize,0)$);
	\filldraw[color=yellow!60] (start)+(0,0) circle (2*\diabulletsize);
	\filldraw[color=yellow!60] (start)+(0,\dianumbasenodes*\diayoffset+1.25*\diayoffset) circle (2*\diabulletsize);

	\coordinate (start) at ($(base)+(\diaxoffset,0)$);
	\filldraw[color=gray!60] ($(start)+(-2*\diabulletsize,0)$) -- ($(start)+(-2*\diabulletsize,\dianumbasenodes*\diayoffset+1.25*\diayoffset)$)
	-- ($(start)+(0,\dianumbasenodes*\diayoffset+1.25*\diayoffset)$)
	-- ++($(0,2*\diabulletsize)$)
	-- ++($(3*\diaxoffset,0)$)
	-- ++($(2*\diabulletsize,-2*\diabulletsize)$)
	-- ($(start)+(3*\diaxoffset,0)+(2*\diabulletsize,0)$)
	-- ++($(-2*\diabulletsize,-2*\diabulletsize)$)
	-- ($(start)+(0,-2*\diabulletsize)$) -- ($(start)+(-2*\diabulletsize,0)$);
	\filldraw[color=gray!60] (start)+(0,0) circle (2*\diabulletsize);
	\filldraw[color=gray!60] ($(start)+(3*\diaxoffset,0)$) circle (2*\diabulletsize);
	\filldraw[color=gray!60] ($(start)+(3*\diaxoffset,\dianumbasenodes*\diayoffset+1.25*\diayoffset)$) circle (2*\diabulletsize);
	\filldraw[color=gray!60] (start)+(0,\dianumbasenodes*\diayoffset+1.25*\diayoffset) circle (2*\diabulletsize);


	\coordinate (start) at ($(base)$);
	\draw (start) -- ($(start)+(0,\dianumbasenodes*\diayoffset-0.1*\diayoffset)$);
	\foreach \n in {1,...,\dianumbasenodes} {
			\draw [fill=blue] ($(start)+(0,\n*\diayoffset-\diayoffset)$) circle (\diabulletsize);
	};
	\node at ($(start)+(0,\dianumbasenodes*\diayoffset+0.75*\diayoffset)$) {$\vdots$};
	\draw [fill=blue] ($(start)+(0,\dianumbasenodes*\diayoffset+1.25*\diayoffset)$) circle (\diabulletsize);
	\node at ($(start)+(0.0*\diaxoffset,-0.75*\diayoffset)$) {\scriptsize $Y$};
	\node at ($(start)+(2.5*\diaxoffset,-0.75*\diayoffset)$) {\scriptsize $\Spc(\Exc{d-1}^{\dname})$};

	\coordinate (start) at ($(base)+(\diaxoffset,0)$);
	\draw (start) -- ($(start)+(0,\dianumbasenodes*\diayoffset-0.1*\diayoffset)$);
	\foreach \n in {0,...,\dianumbasenodesminus} {
			\draw [fill=red] ($(start)+(0,\n*\diayoffset)$) circle (\diabulletsize);
	};
	\node at ($(start)+(0,4*\diayoffset+0.75*\diayoffset)$) {$\vdots$};
	\draw [fill=red] ($(start)+(0,\dianumbasenodes*\diayoffset+1.25*\diayoffset)$) circle (\diabulletsize);

	\coordinate (start) at ($(base)+(2*\diaxoffset,0)$);
	\draw (start) -- ($(start)+(0,\dianumbasenodes*\diayoffset-0.1*\diayoffset)$);
	\foreach \n in {1,...,\dianumbasenodes} {
			\draw [fill=red] ($(start)+(0,\n*\diayoffset-\diayoffset)$) circle (\diabulletsize);
	};
	\node at ($(start)+(0,4*\diayoffset+0.75*\diayoffset)$) {$\vdots$};
	\draw [fill=red] ($(start)+(0,\dianumbasenodes*\diayoffset+1.25*\diayoffset)$) circle (\diabulletsize);

	\coordinate (start) at ($(base)+(3*\diaxoffset,-0.5*\diayoffset)$);
	\foreach \n in {1,...,\dianumbasenodes} {
	\node at ($(start)+(0,\n*\diayoffset)$) {$\hdots$};
	};
	\node at ($(start)+(0,4*\diayoffset+\diayoffset)$) {$\hdots$};

	\coordinate (start) at ($(base)+(4*\diaxoffset,0)$);
	\draw (start) -- ($(start)+(0,\dianumbasenodes*\diayoffset-0.1*\diayoffset)$);
	\foreach \n in {1,...,\dianumbasenodes} {
			\draw [fill=red] ($(start)+(0,\n*\diayoffset-\diayoffset)$) circle (\diabulletsize);
	};
	\node at ($(start)+(0,4*\diayoffset+0.75*\diayoffset)$) {$\vdots$};
	\draw [fill=red] ($(start)+(0,\dianumbasenodes*\diayoffset+1.25*\diayoffset)$) circle (\diabulletsize);
%
\end{tikzpicture}

%% file: exc3-p2-tikz.tex
\def\diabulletsize{0.06cm}
\def\diayoffset{0.45}
\def\diaxoffset{0.45}
\def\dianumbasenodes{4}
\def\dianumbasenodesminus{3}
\def\dianumbasenodesminust{2}
\begin{tikzpicture}
	\coordinate (base) at (0,0);


	\coordinate (start) at (base);
	\filldraw[color=yellow!60] ($(start)+(-2*\diabulletsize,0)$) -- ($(start)+(-2*\diabulletsize,\dianumbasenodes*\diayoffset+1.25*\diayoffset)$)
	-- ($(start)+(2*\diabulletsize,\dianumbasenodes*\diayoffset+1.25*\diayoffset)$)
	-- ($(start)+(2*\diabulletsize,0)$) -- ($(start)+(-2*\diabulletsize,0)$);
	\filldraw[color=yellow!60] (start)+(0,0) circle (2*\diabulletsize);
	\filldraw[color=yellow!60] (start)+(0,\dianumbasenodes*\diayoffset+1.25*\diayoffset) circle (2*\diabulletsize);

	\coordinate (start) at ($(base)+(\diaxoffset,0)$);
	\filldraw[color=cyan!60] ($(start)+(-2*\diabulletsize,0)$) -- ($(start)+(-2*\diabulletsize,\dianumbasenodes*\diayoffset+1.25*\diayoffset)$)
	-- ($(start)+(2*\diabulletsize,\dianumbasenodes*\diayoffset+1.25*\diayoffset)$)
	-- ($(start)+(2*\diabulletsize,0)$) -- ($(start)+(-2*\diabulletsize,0)$);
	\filldraw[color=cyan!60] (start)+(0,0) circle (2*\diabulletsize);
	\filldraw[color=cyan!60] (start)+(0,\dianumbasenodes*\diayoffset+1.25*\diayoffset) circle (2*\diabulletsize);

	\draw ($(base)+(0,\dianumbasenodes*\diayoffset+1.25*\diayoffset)$) -- ($(base)+(\diaxoffset,\dianumbasenodes*\diayoffset+1.25*\diayoffset)$);
	\draw ($(base)+(\diaxoffset,\dianumbasenodes*\diayoffset+1.25*\diayoffset)$) -- ($(base)+(2*\diaxoffset,\dianumbasenodes*\diayoffset+1.25*\diayoffset)$);
	\foreach \n in {0,...,\dianumbasenodesminust} {
		\draw ($(base)+(\diaxoffset,\n*\diayoffset)$) -- ($(base)+(0,\n*\diayoffset+\diayoffset)$);
	};
	\foreach \n in {\dianumbasenodesminus} {
		\draw ($(base)+(\diaxoffset,\n*\diayoffset)$) -- ($(base)+(0.1*\diaxoffset,\n*\diayoffset+0.9*\diayoffset)$);
	};
	\foreach \n in {0,...,\dianumbasenodesminust} {
		\draw ($(base)+(2*\diaxoffset,\n*\diayoffset)$) -- ($(base)+(\diaxoffset,\n*\diayoffset+\diayoffset)$);
	};
	\foreach \n in {\dianumbasenodesminus} {
		\draw ($(base)+(2*\diaxoffset,\n*\diayoffset)$) -- ($(base)+(1.1*\diaxoffset,\n*\diayoffset+0.9*\diayoffset)$);
	};

	\coordinate (start) at ($(base)$);
	\draw (start) -- ($(start)+(0,\dianumbasenodes*\diayoffset-0.1*\diayoffset)$);
	\foreach \n in {1,...,\dianumbasenodes} {
			\draw [fill=blue] ($(start)+(0,\n*\diayoffset-\diayoffset)$) circle (\diabulletsize);
			\node at ($(start)+(-0.7,\n*\diayoffset-\diayoffset)$) {\scriptsize $\cat P([3],\n)$};
	};
	\node at ($(start)+(0,\dianumbasenodes*\diayoffset+0.75*\diayoffset)$) {$\vdots$};
	\draw [fill=blue] ($(start)+(0,\dianumbasenodes*\diayoffset+1.25*\diayoffset)$) circle (\diabulletsize);
	\node at ($(start)+(-0.7,\dianumbasenodes*\diayoffset+1.25*\diayoffset)$) {\scriptsize $\cat P([3],\infty)$};
	\node at ($(start)+(0,-0.75*\diayoffset)$) {\scriptsize $[3]$};

	\coordinate (start) at ($(base)+(\diaxoffset,0)$);
	\draw (start) -- ($(start)+(0,\dianumbasenodes*\diayoffset-0.1*\diayoffset)$);
	\foreach \n in {0,...,\dianumbasenodesminus} {
			\draw [fill=green] ($(start)+(0,\n*\diayoffset)$) circle (\diabulletsize);
	};
	\node at ($(start)+(0,4*\diayoffset+0.75*\diayoffset)$) {$\vdots$};
	\draw [fill=green] ($(start)+(0,\dianumbasenodes*\diayoffset+1.25*\diayoffset)$) circle (\diabulletsize);
	\node at ($(start)+(0,-0.75*\diayoffset)$) {\scriptsize $[2]$};

	\coordinate (start) at ($(base)+(2*\diaxoffset,0)$);
	\draw (start) -- ($(start)+(0,\dianumbasenodes*\diayoffset-0.1*\diayoffset)$);
	\foreach \n in {1,...,\dianumbasenodes} {
			\draw [fill=red] ($(start)+(0,\n*\diayoffset-\diayoffset)$) circle (\diabulletsize);
			\node at ($(start)+(0.65,\n*\diayoffset-\diayoffset)$) {\scriptsize $\cat P([1],\n)$};
	};
	\node at ($(start)+(0,4*\diayoffset+0.75*\diayoffset)$) {$\vdots$};
	\draw [fill=red] ($(start)+(0,\dianumbasenodes*\diayoffset+1.25*\diayoffset)$) circle (\diabulletsize);
	\node at ($(start)+(0.65,\dianumbasenodes*\diayoffset+1.25*\diayoffset)$) {\scriptsize $\cat P([1],\infty)$};
	\node at ($(start)+(0,-0.75*\diayoffset)$) {\scriptsize $[1]$};

\end{tikzpicture}